\documentclass[aihp]{imsart}

\RequirePackage{amsthm,amsmath,amsfonts,amssymb}
\RequirePackage[numbers]{natbib}
\RequirePackage[colorlinks,citecolor=blue,urlcolor=blue]{hyperref}

\startlocaldefs
\theoremstyle{plain}
\newtheorem{theorem}{Theorem}
\newtheorem{lemma}{Lemma}

\usepackage{tikz}
\usetikzlibrary{patterns}
\usetikzlibrary{decorations.pathreplacing}
\usetikzlibrary{arrows,decorations.markings}

\newcommand{\abs}[1]{\left| #1 \right|}
\newcommand{\norme}[1]{\left\| #1 \right\|}
\newcommand{\N}{\mathbb{N}}
\newcommand{\Z}{\mathbb{Z}}
\newcommand{\R}{\mathbb{R}}
\newcommand{\restrict}[1]{\raisebox{-.5ex}{$|$}_{#1}} 
\newcommand{\Proba}{\mathbb{P}}
\newcommand{\Ent}[1]{\left\lfloor #1\right\rfloor}
\newcommand{\Ceil}[1]{\left\lceil #1\right\rceil}
\newcommand{\demi}{\frac{1}{2}}
\newcommand{\eqninfty}{\stackrel{n\to\infty}{\sim}}
\newcommand\numberthis{\addtocounter{equation}{1}\tag{\theequation}}
\newcommand{\quadet}{\quad\text{and}\quad}
\newcommand{\qquadet}{\qquad\text{and}\qquad}

\newcommand{\quadavec}{\quad\text{with}\quad}

\newcommand{\quadimplique}{\quad\Rightarrow\quad}

\renewcommand{\thesection}{\arabic{section}}
\newcommand{\guillemets}[1]{``#1''}
\newcommand{\calE}{\mathcal{E}}
\newcommand{\calS}{\mathcal{S}}
\newcommand{\calC}{\mathcal{C}}
\newcommand{\calF}{\mathcal{F}}
\newcommand{\calT}{\mathcal{T}}
\newcommand{\calP}{\mathcal{P}}
\newcommand{\calA}{\mathcal{A}}
\newcommand{\s}{\mathfrak{s}}
\newcommand{\PCM}{\textbf{P}^{A}}
\newcommand{\boxO}{B_0}
\newcommand{\boxU}{B_0'}
\newcommand{\boxI}{B_1}
\newcommand{\boxII}{B_2}
\newcommand{\boxIII}{B_3}
\newcommand{\Pl}{\calP^{\lambda}}
\newcommand{\Pinf}{\calP^{\infty}}
\newcommand{\PlA}{\calP^{\lambda,A}}
\newcommand{\Pm}{\calP_{\mu}}
\newcommand{\Pml}{\calP_{\mu}^{\lambda}}
\newcommand{\Em}{\calE_{\mu}}
\newcommand{\El}{\calE^{\lambda}}
\newcommand{\Eml}{\calE_{\mu}^{\lambda}}
\newcommand{\PARW}{\Proba_{\mu}^{\lambda}}
\newcommand{\torus}{{\Z_n^d}}
\newcommand{\sle}{node[blue!50!black]{s}}
\newcommand{\act}{node[red!50!black]{1}}

\endlocaldefs

\begin{document}

\begin{frontmatter}

\title{Active Phase for Activated Random Walks on the Lattice in all Dimensions}
\runtitle{Active phase for ARW on~$\mathbb{Z}^d$}

\begin{aug}
\author[A]{\inits{N. F.}\fnms{Nicolas}~\snm{Forien}\ead[label=e1]{nicolas.forien@uniroma1.it}}
\and
\author[B]{\inits{A. G.}\fnms{Alexandre}~\snm{Gaudilli\`ere}\ead[label=e2]{alexandre.gaudilliere@math.cnrs.fr}}
\address[A]{Aix Marseille Univ, CNRS, I2M, Marseille, France, Sapienza Universit\`a di Roma, Dipartimento di Matematica, Roma, Italy\printead[presep={,\ }]{e1}}

\address[B]{Aix Marseille Univ, CNRS, I2M, Marseille, France\printead[presep={,\ }]{e2}}
\end{aug}

\begin{abstract}
We show that the critical density of the Activated Random Walk model on~$\Z^d$ is strictly less than one when the sleep rate~$\lambda$ is small enough, and tends to~$0$ when~$\lambda\to 0$, in any dimension~$d\geqslant 1$.
As far as we know, the result is new for~$d=2$.

We prove this by showing that, for high enough density and small enough sleep rate, the stabilization time of the model on the~$d$-dimensional torus is exponentially large.
To do so, we fix the the set of sites where the particles eventually fall asleep, which reduces the problem to a simpler model with density one.
Taking advantage of the Abelian property of the model, we show that the stabilization time stochastically dominates the escape time of a one-dimensional random walk with a negative drift.
We then check that this slow phase for the finite volume dynamics implies the existence of an active phase on the infinite lattice.
\end{abstract}

\begin{abstract}[language=french]
Nous démontrons que la densité critique du modèle des Marches Aléatoires Activées sur~$\Z^d$ est strictement inférieure à~$1$ quand le taux d'endormissement~$\lambda$ est suffisamment petit, et tend vers~$0$ quand~$\lambda\to 0$, en toute dimension~$d\geqslant 1$.
À notre connaissance, le résultat est nouveau pour~$d=2$.

Nous obtenons ce résultat en prouvant que, pour une densité suffisamment élevée et un taux d'endormissement suffisamment petit, le temps de stabilisation du modèle sur le tore en dimension~$d$ est exponentiellement grand.
Pour cela, nous fixons l'ensemble des sites sur lesquels les particules s'endorment, ce qui réduit le problème à un modèle plus simple avec densité~1.
En utilisant la propriété d'Abélianité du modèle, nous montrons que le temps de stabilisation domine stochastiquement le temps d'atteinte de 0 pour une marche aléatoire en dimension~1 avec une dérive négative.
Nous vérifions ensuite que cette phase de stabilisation lente pour la dynamique en volume fini implique l'existence d'une phase active sur le réseau infini.
\end{abstract}

\begin{keyword}[class=MSC]
\kwd[Primary ]{60K35}
\kwd[; secondary ]{82B26}
\end{keyword}

\begin{keyword}
\kwd{Activated random walks}
\kwd{phase transition}
\kwd{self-organized criticality}
\end{keyword}

\end{frontmatter}


\section{Introduction}

The Activated Random Walk model is a model of interacting particles which attracted much interest since the seminal works of Rolla, Sidoravicius and Dickman~\cite{Rolla08,RS12,DRS10}.
This model emerged as a modification of another process suggested by R. Durrett in 1996 and known as the frog model~\cite{AMP02a,AMP02b,GS09,HJJ17,JJ18}, and it can also be seen as a particular case of some models for the spread of a rumor or an infection~\cite{KS05}.

\subsection{Informal definition of the model}
\label{sectionDefARW}

Informally, the Activated Random Walk model is defined as follows.
Let us consider a vertex-transitive and locally finite graph~$G=(V,\,E)$, with, on each site of the graph, a certain number of particles, which may be either active or sleeping.

Active particles perform independent continuous-time random walks on the graph, with jump rate~1, meaning that each active particle is equipped with a Poisson clock of intensity~1 and, each time this clock rings, this particle jumps to a uniformly chosen neighbouring site of the graph.

In addition to this, when an active particle is alone on a site, it falls asleep with rate~$\lambda\in[0,\infty]$.
Thus, each particle is equipped with a second Poisson clock, with intensity~$\lambda$: when this clock rings, if there is no other particle on the same site, the active particle falls asleep.
Sleeping particles stop moving, until another active particle arrives on the same site.
When this happens, the sleeping particle is waken up and turns back into the active state.

The initial configuration of particles is assumed to be distributed according to a translation-ergodic distribution, with no sleeping particles and an average density of active particles~$\mu\in(0,\infty)$.
Such a model is well defined if the graph is finite, and the above informal definition can also be made rigorous on~$\Z^d$~\cite{Rolla20}.

One motivation to study this model is its connection with the concept of self-organized criticality~\cite{BTW87,Jensen98,Dhar06,These,HALWBDW20}, which aims to describe some physical systems which present a critical-like behaviour without the need to finely tune their parameters to precise values (unlike in an ordinary phase transition, where the critical behaviour is only observed at the critical point).

The connection with self-organized criticality lies in the following variant of the model.
Consider the model on a finite box where particles are killed when they jump out of the box.
At each step, we add one active particle to a uniformly chosen site, and we let the box evolve until it reaches a stable configuration.
It is conjectured that, when the size of this box tends to infinity, the average density of particles converges to the critical density of the model on~$\Z^d$, as defined below (see for example~\cite{BGH18} for a presentation of this conjecture).
This property attracted some attention on the Activated Random Walk model, with the hope that this model would be more analytically tractable than other related models showing self-organized criticality, like the Stochastic Sandpile model~\cite{ST17} or the Manna model~\cite{Manna91}.

\subsection{Some previous results on the phase transition for the model on the lattice}

The Activated Random Walk model on~$\Z^d$ presents what is called an absorbing state phase transition.
Namely, there is a critical curve in the plane~$(\mu,\,\lambda)$ separating two different regimes: a fixating phase where, almost surely, the model eventually reaches an absorbing state, and an active phase where the model never stabilizes.
Here, the absorbing states of the system, also called stable configurations, are the states where all the particles are sleeping.

We say that the system fixates if each site is visited only finitely many times.
In this case, the configuration on any finite set eventually becomes stable, with only sleeping particles and empty sites.
Otherwise, we say that the system stays active.
The following key result of Rolla, Sidoravicius and Zindy establishes the existence of a universal phase transition for the model on~$\Z^d$:

\begin{theorem}[\cite{RSZ19}]
\label{thmPhaseTransition}
For every dimension~$d\geqslant 1$, and for every sleep rate~$\lambda\in(0,\infty]$, there exists~$\mu_c(\lambda)$ such that, for every translation-ergodic initial distribution with no sleeping particles and an average density of active particles~$\mu$, the Activated Random Walk model on~$\Z^d$ with sleep rate~$\lambda$ almost surely fixates if~$\mu<\mu_c(\lambda)$, whereas it almost surely stays active if~$\mu>\mu_c(\lambda)$.
\end{theorem}

This threshold~$\mu_c(\lambda)$ is called the critical density of the model.
These last years, much effort has been directed towards finding upper and lower bounds on this critical density.
As a first step, it was proved in~\cite{Shellef10,AGG10} that~$\mu_c(\lambda)\leqslant 1$ for any value of the sleep rate, by showing that the density of particles is conserved over time.

Then, a natural question is: do we have~$\mu_c(\lambda)\in(0,1)$ for any finite value of the sleep rate~$\lambda$ in any dimension?
It is known that~$\mu_c(\lambda)\geqslant\lambda/(1+\lambda)$ in wide generality~\cite{RS12,ST17,ST18}, but this bound is sharp only when the jump distribution is totally asymmetric~\cite{CRS14}.

When the jumps are biased, it is known that~$\mu_c(\lambda)<1$ for any finite value of the sleep rate~$\lambda$ and that~$\mu_c(\lambda)\to 0$ when~$\lambda\to 0$, in any dimension~\smash{$d\geqslant 1$}~\cite{Taggi16,RT18}.

For unbiased jumps, it was proved that~$\mu_c(\lambda)\to 0$ when~$\lambda\to 0$ in dimension~1 by~\cite{BGH18} and in dimensions~$d\geqslant 3$ in~\cite{ST18}, and it was also shown that~$\mu_c(\lambda)<1$ for any finite value of the sleep rate~$\lambda$, in dimension~1~\cite{HRR20} and in dimensions~\smash{$d\geqslant 3$}~\cite{Taggi19}.
In dimension~1, a more precise bound on the critical density showing that~\smash{$\mu_c(\lambda)=O\big(\sqrt{\lambda}\big)$} when~$\lambda\to 0$ was obtained in~\cite{ARS19}.

\subsection{Existence of an active phase in all dimensions}

For the unbiased case on~$\Z^2$, the best bound available up to now was~$\mu_c(\lambda)\leqslant 1$, and we did not know any choice of~$\mu<1$ and~$\lambda>0$ such that there is no fixation at~$(\mu,\,\lambda)$, which would imply that~$\mu_c(\lambda)<1$ for sufficiently low sleep rate.

Our main result stated below shows that~$\mu_c(\lambda)\to 0$ when~$\lambda\to 0$ in any dimension~$d\geqslant 1$ which, in particular, answers the open problem of the existence of a non-trivial active phase on~$\Z^2$.

\begin{theorem}
\label{mainThm}
In any dimension~$d\geqslant 1$, we have
$$\lim\limits_{\lambda\to 0}\mu_c(\lambda)\ =\ 0\,.$$
More precisely, there exists a constant~$\kappa_d>0$ such that,
for every~$\mu\in(0,1)$ and~$\lambda>0$ which satisfy
\begin{equation}
\label{conditionMuLambda}
\kappa_d(2d\lambda)^\mu\ <\ \mu^\mu (1-\mu)^{1-\mu}\,,
\end{equation}
for every translation-ergodic initial distribution with no sleeping particles and an average density of active particles~$\mu$, the Activated Random Walk model on~$\Z^d$ with sleep rate~$\lambda$ almost surely stays active.

Thus, for any initial active particle density~$\mu\in(0,1)$, there exists~$\lambda>0$ such that, for every translation-ergodic initial distribution with no sleeping particles and an average density of active particles~$\mu$, the Activated Random Walk model on~$\Z^d$ with sleep rate~$\lambda$ almost surely stays active, and for all~\smash{$\lambda>0$} small enough, there exists~$\mu<1$ such that, for every translation-ergodic initial distribution with no sleeping particles and an average density of active particles~$\mu$, the Activated Random Walk model on~$\Z^d$ with sleep rate~$\lambda$ almost surely stays active.
\end{theorem}

Our result that~$\mu_c(\lambda)\to 0$ when~$\lambda\to 0$ was already known for the transient case~$d\geqslant 3$ and the special case~$d=1$, but our proof works in any dimension~\smash{$d\geqslant 1$}.

Note that we only prove that~$\mu_c(\lambda)<1$ for~$\lambda<1/(2d\kappa_d)$, leaving open the conjecture that~$\mu_c(\lambda)<1$ for any finite value~$\lambda$ of the sleep rate.
It might be possible to improve our technique to obtain this, but our efforts in this direction have not succeeded so far (see paragraph~\ref{sectionEchelleOurs} for more comments on this).

\subsection{Slow phase for the finite volume dynamics}

In order to prove Theorem~\ref{mainThm}, we first study the model on the~$d$-dimensional torus~$\torus=(\Z/n\Z)^d$.
Recall that a configuration is said to be stable if all the particles are sleeping.
Stable configurations are absorbing states of the model, meaning that if the dynamic reaches such a configuration, then it cannot change anymore.
Hence, on a finite graph, if the initial number of particles is not greater than the number of sites, the system eventually fixates almost surely.
Therefore, instead of the probability for the system to fixate, we are interested in the time it takes for the system to fixate, i.e., to reach a stable configuration.

Denoting this random time~$\calT_n$, and writing~$\PARW$ for the probability measure relative the Activated Random Walk model on the torus with sleep rate~$\lambda$ starting from an initial configuration with independent Poisson random numbers of active particles with mean~$\mu$ on each site, we prove the following result:

\begin{theorem}
\label{thmSlowPhase}
In any dimension~$d\geqslant 1$, for every~$\mu\in(0,1)$ and~$\lambda>0$ which satisfy the condition~(\ref{conditionMuLambda}), the fixation time~$\calT_n$ of the Activated Random Walk model on the torus~$\torus=(\Z/n\Z)^d$ satisfies
\begin{equation}
\label{tempsExp}
\exists\,c>0\,,\quad
\forall\,n\geqslant 1\,,\qquad
\PARW\Big(\,\calT_n\,<\,e^{cn^d}\,\Big)
\ <\ e^{-cn^d}\,.
\end{equation}
\end{theorem}

Note that the result remains true if~$\calT_n$ is replaced with the number of times a particle jumps or falls asleep (later on we will call this the number of topplings) until the system fixates (see Lemma~\ref{lemmaSlowPhaseTopplings}).

To the best of our knowledge, this slow phase for finite systems was only studied up to now in the one-dimensional setting.
In~\cite{BGHR19}, it is shown that~(\ref{tempsExp}) holds on~$\Z/n\Z$ as soon as~$\mu>\mu_c(\lambda)$.

\subsection{Slow phase on the torus implies active phase in infinite volume}

Once Theorem~\ref{thmSlowPhase} is proved, Theorem~\ref{mainThm} is obtained through the following result, which relates the existence of a slow phase for the dynamics on the torus to the existence of an active phase for the model on the infinite lattice~$\Z^d$.

\begin{theorem}
\label{thmLink}
In any dimension~$d\geqslant 1$, for every sleep rate~$\lambda\in(0,\,\infty)$ and every~$\mu<\mu_c(\lambda)$, the property~(\ref{tempsExp}) about the fixation time~$\calT_n$ of the model on the torus~$\torus$ does not hold.
\end{theorem}

This Theorem is a very weak version of a conjecture mentioned in section~6 of~\cite{Rolla20}, which predicts that for every density~\smash{$\mu<\mu_c(\lambda)$}, there should exist~$\kappa>0$ and~$\delta>0$ such that, for~$n$ large enough,
$$\PARW\Big(\,\calT_n\,\leqslant\,\kappa n^d\,\Big)
\ \geqslant\ 1-n^{-\delta}\,.$$
Combined with Theorem~\ref{thmSlowPhase}, the above Theorem entails that~$\mu\geqslant\mu_c(\lambda)$ for every~$\mu\in(0,\,1)$ and~$\lambda>0$ satisfying the condition~(\ref{conditionMuLambda}) and, this condition being a strict inequality, it implies that~$\mu>\mu_c(\lambda)$ for every~$\mu\in(0,\,1)$ and~$\lambda>0$ satisfying this condition which, given Theorem~\ref{thmPhaseTransition}, yields Theorem~\ref{mainThm}.

\subsection{Key ingredients and organization of the article}

In section~\ref{sectionDiaconis}, we recall the site-wise representation of the Activated Random Walk model and the the properties which follow from it, including the monotonicity property and the Abelian property, which are key ingredients in our method.
In section~\ref{sectionProofSlowPhase}, we prove Theorem~\ref{thmSlowPhase} about the fixation time on the torus, following the general strategy presented below.
We conclude in section~\ref{sectionProofThmLink} with the proof of Theorem~\ref{thmLink} which, together with Theorems~\ref{thmSlowPhase} and~\ref{thmPhaseTransition}, implies our main result, the existence of an active phase on~$\Z^d$ for all~$d\geqslant 1$ (Theorem~\ref{mainThm}).

Let us sketch below the main elements of the proofs of our two main Theorems.

\subsubsection{Sketch of the proof of Theorem~\ref{thmSlowPhase}}
\label{sectionSketchThmSlowPhase}

We first sketch the proof of Theorem~\ref{thmSlowPhase}, that is to say, for~$\lambda$ and~$\mu$ satisfying the condition~(\ref{conditionMuLambda}), we want to show that the stabilization time on the torus is exponentially large with overwhelming probability.
All the intuitive facts claimed below are formally proved in section~\ref{sectionProofSlowPhase}.

\paragraph{Fixing the initial number of particles:}
Taking~$\mu'$ slightly lower than~$\mu$, but such that the condition~(\ref{conditionMuLambda}) still holds for~$\lambda$ and~$\mu'$, we may reason conditionally on the fact that there are at least~$\mu' n^d$ particles in the initial configuration (which happens with overwhelming probability).
Then, using the monotonicity result of Lemma~\ref{lemmaMonotonicity}, it is enough to show that the stabilization time is exponential starting from any deterministic initial configuration with exactly~$k=\Ceil{\mu' n^d}$ active particles.

\paragraph{Fixing the set of sleeping sites:}
This is the first key ingredient of our proof.
Once the number~$k$ of particles is fixed, the probability that these~$k$ particles fixate within a given time can be written as a sum over all the subsets~$A\subset\torus$ with cardinality~$k$ of the probability that they fixate on~$A$ within a given time.
Thus, we fix an arbitrary set~$A$ with cardinality~$k$, and we try to evaluate the probability that the particles all fall asleep on~$A$ within a given time, that is to say, that we eventually reach the final stable configuration with one sleeping particle on each site of~$A$.
We will then have a combinatorial factor~\smash{$\binom{n^d}{k}$} relative to the choice of~$A$, but we will be able to compensate it by choosing~$\lambda$ small enough.

\paragraph{Forbidding particles to sleep elsewhere:}
Then, a second key idea is to modify the model by forbidding particles to fall asleep outside of~$A$.
Indeed, if the particles eventually fixate on~$A$, it means that the sleep instructions which happened outside of~$A$ have not been useful, because all the particles which fell asleep outside of~$A$ have eventually been waken up by another particle.
Therefore, this modification of the model only increases the probability that the particles fixate on~$A$ within a given time (see section~\ref{sectionSleepless} for a formal proof of this intuitive fact).

\paragraph{Taking advantage of the Abelian property of the site-wise representation:}
Instead of reasoning with the continuous time with Poisson clocks, we bound the number of topplings, i.e., the number of times a particle jumps or tries to fall asleep until the system stabilizes.
If we show that this number of topplings is exponentially large, then it easily follows that the fixation time of the continuous-time model is also exponentially large.

To evaluate this number of topplings, we use the site-wise representation of the model (see section~\ref{sectionSitewise}) and we take advantage of its Abelian property (Lemma~\ref{lemmaAbelian}), which says that the number of topplings and the final configuration do not depend on the order with which the sites are toppled.
Thus, our goal is to find a toppling strategy which enables us to show that, with overwhelming probability, we can perform an exponentially large number of topplings before the system stabilizes on~$A$ (remember that, with the above modification of the model, the only stable configuration which can be reached is the configuration with one sleeping particle on each site of~$A$).

\paragraph{Having one particle on each site of the settling set:}
In order to stabilize all the particles on the sites of~$A$, as a first step we may let all the particles move until they arrive on an unoccupied site of~$A$.
Hence, the number of topplings necessary to fixate on~$A$ starting from any initial configuration with exactly~$k=|A|$ particles stochastically dominates the number of topplings necessary to fixate on~$A$ starting with one active particle on each site of~$A$ (see section~\ref{sectionWalkToA} for a formal proof of this).
Thus, in what follows we may as well assume that we start with the initial configuration with one active particle on each site of~$A$.

\paragraph{Toppling the sites in the right order:}
Then, our third key idea is to order the sites of~$A$ by writing~$A=\{x_1,\,\ldots,\,x_k\}$ in a way such that the distances~$d(x_j,\,x_{j+1})$ are not too big (see section~\ref{sectionOrder}).
Given this order, our toppling strategy is composed of a certain number of steps ending with exactly one particle (either active or sleeping) on each site of~$A$.

At each step, we start by toppling a site~$x_j\in A$ which contains an active particle and such that for all~$i>j$, there is also an active particle on~$x_i$, as drawn on figure~\ref{figTopplingProcedure} (at the first step, we topple~$x_1$).
With probability~$\lambda/(1+\lambda)$, the particle falls asleep on~$x_j$ before moving.
In this case, we proceed to the next step, where we topple the site~$x_{j+1}$.
Otherwise we let the particle walk with successive topplings until it comes back to its starting point~$x_j$ (remember that it is not allowed to fall asleep anywhere else, all the other sites of~$A$ being occupied).
While walking from~$x_j$ back to~$x_j$, this active particle may wake up some of the particles which were sleeping on previous sites of~$A$.
In particular, if it visits~$x_{j-1}$ before returning to~$x_j$, we may choose, in the next step (i.e., once the particle has returned on~$x_j$), to topple the site~$x_{j-1}$, where the particle is necessarily awake.
Otherwise, we topple again~$x_j$ in the next step.
This procedure ends when, toppling the last particle~$x_k$, it falls asleep before moving.
At this point, some of the previous particles may still be active, but we obtain a lower bound for the number of topplings necessary to stabilize the configuration (see section~\ref{sectionDynamic} for the precise description of this toppling strategy).

\begin{figure}
\begin{center}
\begin{tikzpicture}[scale=1.1]
\draw (6.2,4.3) node{During each step, the particle at site~$x_j$ can:};
\draw[gray,dotted,thick] (0.4,0.6) \sle -- (0.9,0.7) \sle -- (1.2,0.8) \act -- (2,0.6) \sle -- (2.1,1) \act -- (2.9,1.1) \act -- (3.4,1.7) \sle -- (4,1.6) (0,1.6) -- (0.3,1.55) \sle -- (1,1.7) \act -- (1.8,2) \sle -- (2.3,1.9) \sle -- (2.5,2.5) node[red!50!black,scale=0.8]{\textbf{1$\to${\color{blue!50!black}s}}} -- (1.6,2.6) \act -- (1,2.7) \act -- (0.4,2.5) \act -- (0.4,3.1) \act -- (0.95,3.5) \act -- (2.2,3.3) \act -- (2.6,3.8) \act -- (2.8,4) (2.8,0) -- (3.1,0.3) \act -- (3.5,0.4) \act -- (3.6,0) (3.6,4) -- (3.8,3.2) \act -- (3.6,2.9) \act;
\draw (0,0) -- (0,4) -- (4,4) -- (4,0) -- cycle;
\draw (0.3,0.2) node[scale=0.7]{$x_1$};
\draw[->] (0.32,0.3) -- (0.38,0.5);
\draw (2.7,2.9) node{$x_j$};
\draw[->] (2.67,2.8) -- (2.53,2.6);
\draw (2.2,1.5) node{$x_{j-1}$};
\draw[->] (2.22,1.6) -- (2.28,1.8);
\draw (3.7,2.5) node{$x_k$};
\draw[->] (3.68,2.6) -- (3.62,2.8);
\draw (2,-0.3) node{fall asleep,};
\draw (2,-0.6) node{with probability};
\draw (2,-1) node{$\frac{\lambda}{1+\lambda}\,,$};

\begin{scope}[xshift=4.2cm]
\draw[thick,red!70!orange,->] (2.6,2.4) to[out=-30,in=60] (3.2,2) to[out=-120,in=-70] (2.33,1.8) (2.27,2) to[out=110,in=-160] (2.4,2.43);
\draw[gray,dotted,thick] (0.4,0.6) \sle -- (0.9,0.7) \sle -- (1.2,0.8) \act -- (2,0.6) \sle -- (2.1,1) \act -- (2.9,1.1) \act -- (3.4,1.7) \sle -- (4,1.6) (0,1.6) -- (0.3,1.55) \sle -- (1,1.7) \act -- (1.8,2) \sle -- (2.3,1.9) node[red!50!black,scale=0.8]{\textbf{{\color{blue!50!black}s$\to$}1}} -- (2.5,2.5) node[red!50!black,scale=0.8]{\textbf{1}} -- (1.6,2.6) \act -- (1,2.7) \act -- (0.4,2.5) \act -- (0.4,3.1) \act -- (0.95,3.5) \act -- (2.2,3.3) \act -- (2.6,3.8) \act -- (2.8,4) (2.8,0) -- (3.1,0.3) \act -- (3.5,0.4) \act -- (3.6,0) (3.6,4) -- (3.8,3.2) \act -- (3.6,2.9) \act;
\draw (0,0) -- (0,4) -- (4,4) -- (4,0) -- cycle;
\draw (0.3,0.2) node{$x_1$};
\draw[->] (0.32,0.3) -- (0.38,0.5);
\draw (2.7,2.9) node{$x_j$};
\draw[->] (2.67,2.8) -- (2.53,2.6);
\draw (2.2,1.5) node{$x_{j-1}$};
\draw[->] (2.22,1.6) -- (2.28,1.8);
\draw (3.7,2.5) node{$x_k$};
\draw[->] (3.68,2.6) -- (3.62,2.8);
\draw (2,-0.3) node{wake up the particule on~$x_{j-1}$,};
\draw (2,-0.6) node{with probability};
\draw (2,-1) node{$\frac{1}{1+\lambda}P_{x_{j}}\big(T_{x_{j-1}}<T_{x_j}^+\big)\,,$};
\end{scope}

\begin{scope}[xshift=8.4cm]
\draw[gray,dotted,thick] (0.4,0.6) \sle -- (0.9,0.7) \sle -- (1.2,0.8) \act -- (2,0.6) \sle -- (2.1,1) \act -- (2.9,1.1) \act -- (3.4,1.7) \sle -- (4,1.6) (0,1.6) -- (0.3,1.55) \sle -- (1,1.7) \act -- (1.8,2) \sle -- (2.3,1.9) \sle -- (2.5,2.5) node[red!50!black,scale=0.8]{\textbf{1}} -- (1.6,2.6) \act -- (1,2.7) \act -- (0.4,2.5) \act -- (0.4,3.1) \act -- (0.95,3.5) \act -- (2.2,3.3) \act -- (2.6,3.8) \act -- (2.8,4) (2.8,0) -- (3.1,0.3) \act -- (3.5,0.4) \act -- (3.6,0) (3.6,4) -- (3.8,3.2) \act -- (3.6,2.9) \act;
\draw[thick,red!70!orange,->] (2.6,2.4) to[out=-30,in=-110] (3.3,2.1) to[out=70,in=20] (2.6,2.53);
\draw (0,0) -- (0,4) -- (4,4) -- (4,0) -- cycle;
\draw (0.3,0.2) node{$x_1$};
\draw[->] (0.32,0.3) -- (0.38,0.5);
\draw (2.7,2.9) node{$x_j$};
\draw[->] (2.67,2.8) -- (2.53,2.6);
\draw (2.2,1.5) node{$x_{j-1}$};
\draw[->] (2.22,1.6) -- (2.28,1.8);
\draw (3.7,2.5) node{$x_k$};
\draw[->] (3.68,2.6) -- (3.62,2.8);
\draw (2,-0.3) node{none of these,};
\draw (2,-0.6) node{with probability};
\draw (2,-1) node{$\frac{1}{1+\lambda}P_{x_{j}}\big(T_{x_j}^+<T_{x_{j-1}}\big)\,.$};
\end{scope}
\end{tikzpicture}
\end{center}
\caption{\label{figTopplingProcedure}
At each step of the toppling procedure used in the proof of Theorem~\ref{thmSlowPhase}, the particle at~$x_j$ either falls asleep (left), or visits the site~$x_{j-1}$, waking the particle there if it was sleeping (middle), or goes back to~$x_j$ without visiting~$x_{j-1}$ (right). In this last case, the particle may wake up other particles at sites~$x_i$ with~$i<j-1$, but we do not care about this.
Sleeping particles are denoted by~'s', while active particles are denoted by~'1'.
}
\end{figure}
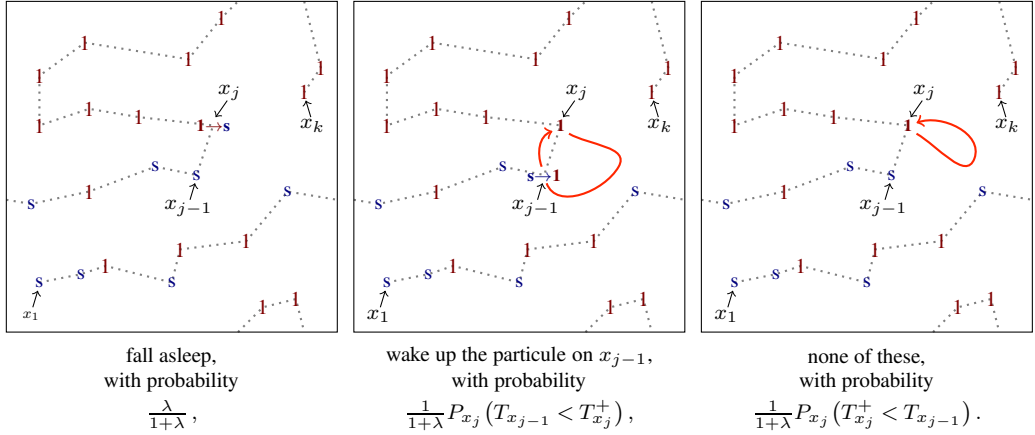

\paragraph{Reduction to a one-dimensional random walk with a negative drift:}
The toppling strategy described above is equivalent to a simplification of the model where, for each~$j$, the particle at~$x_j$ is only allowed to wake up the particle at~$x_{j-1}$, and we ignore collisions with other sleeping particles.
Intuitively, this modification can only increase the probability that the particles fixate on~$A$ within a given time.
In this modified model, at each step, either we add a sleeping particle, or we wake up the preceding particle, or nothing changes.
Thus, the dynamics is reduced to a random walk on the set of integers~$\{1,\,\ldots,\,k+1\}$ because if~$j$ particles are sleeping, it can only be on the first~$j$ sites of~$A$.

At this point, it only remains to check that this random walk (represented on figure~\ref{figRW}) has a negative drift.
To compute this drift, we combine two geometric estimates.
The first estimate, proved in section~\ref{sectionRWtorus}, shows that, for any two distinct points~$x,\,y\in\torus$, the probability that a symmetric random walk started at~$x$ reaches~$y$ before returning to~$x$ satisfies
$$P_{x}\big(T_{y}<T_{x}^+\big)
\ \geqslant\ \frac{1}{2d\,d(x,\,y)}\,,$$
where~$d(x,\,y)$ is the graph distance between~$x$ and~$y$.
The second estimate, obtained in section~\ref{sectionSommeLog}, translates the constraint of the finiteness of the available space in the torus, saying that, if the order~$A=\{x_1,\,\ldots,\,x_k\}$ is well chosen, then
$$\sum_{j=1}^{k-1} \ln\,d(x_j,\,x_{j+1})
\ =\ O\big(n^d\big)\,.$$
Combining these two estimates, we show in section~\ref{sectionMC} that the random walk has a negative drift, which makes the reaching time of~$k+1$ exponentially small with high probability.

We then conclude in three last steps.
In section~\ref{sectionDeterministicEta}, we prove that the number of topplings is exponentially large but starting from a deterministic initial configuration, checking that the combinatorial factor coming from the choice of the set~$A$ can be outweighed by the drift if~$\lambda$ is small enough.
In section~\ref{sectionSlowPhaseTopplings}, we extend this result to an i.i.d.\ Poisson initial distribution, obtaining a version of Theorem~\ref{thmSlowPhase} where the stabilization time~$\calT_n$ is replaced with the number of topplings, before eventually proving this very Theorem in section~\ref{sectionProofSlowPhaseFinal}.

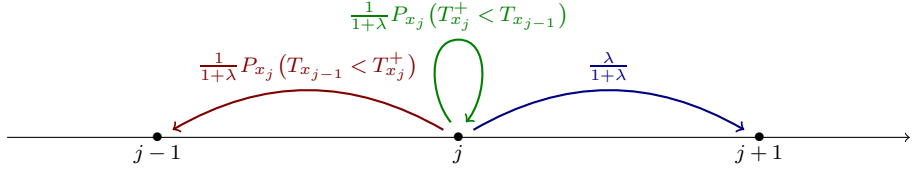
\begin{figure}
\begin{center}
\begin{tikzpicture}
\draw[->] (0,0) -- (12,0);
\draw (2,0) node{$\bullet$} node[below]{$j-1$} (6,0) node{$\bullet$} node[below]{$j$} (10,0) node{$\bullet$} node[below]{$j+1$};
\draw[thick,->,blue!50!black] (6.2,0.1) to[bend left] node[midway,above]{$\frac{\lambda}{1+\lambda}$} (9.8,0.1);
\draw[thick,->,red!50!black] (5.8,0.1) to[bend right] node[midway,above]{$\frac{1}{1+\lambda} P_{x_{j}}\big(T_{x_{j-1}}<T_{x_j}^+\big)$} (2.2,0.1);
\draw[thick,->,green!50!black] (5.9,0.2) to[out=130,in=180] (6,1.3) node[above]{$\frac{1}{1+\lambda} P_{x_{j}}\big(T_{x_j}^+<T_{x_{j-1}}\big)$} to[out=0,in=50]  (6.1,0.2);
\end{tikzpicture}
\end{center}
\caption{\label{figRW}
The one-dimensional random walk which governs the toppling procedure in the proof of Theorem~\ref{thmSlowPhase}.
}
\end{figure}

\paragraph{Summary:}
To put it in a nutshell, by fixing a sleeping set~$A\subset\torus$, we are able to reduce the dynamics to the case of density one (there is exactly one particle on each site of~$A$, either active or sleeping), but on a different graph which is the trace graph on~$A$.
We then take advantage of the abelianity and monotonicity properties of the model to choose the order with which the sites are toppled, which allows us to control the fixation time of the model with a one-dimensional random walk.
The geometric constraints on the set~$A$ then allow us to control the drift of this random walk, and therefore to show that it takes an exponential time to fixate the particles on~$A$.
Because this drift depends on~$\lambda$, while the combinatorial factor~\smash{$\binom{n^d}{k}$} does not, we are able to compensate it by choosing~$\lambda$ sufficiently small.

\subsubsection{Sketch of the proof of Theorem~\ref{thmLink}}
\label{sectionSketchThmLink}

We now present our strategy to prove Theorem~\ref{thmLink}, which says that if~$\mu<\mu_c(\lambda)$, then the stabilization time on the torus~$\torus$ cannot be exponentially large in the sense of property~(\ref{tempsExp}).
Thus, we fix~$c>0$ and we show that, for~$n$ large enough,
\begin{equation}
\label{aimThm4}
\PARW\Big(\,\calT_n\,<\,e^{cn^d}\,\Big)
\ \geqslant\ e^{-cn^d}\,.
\end{equation}

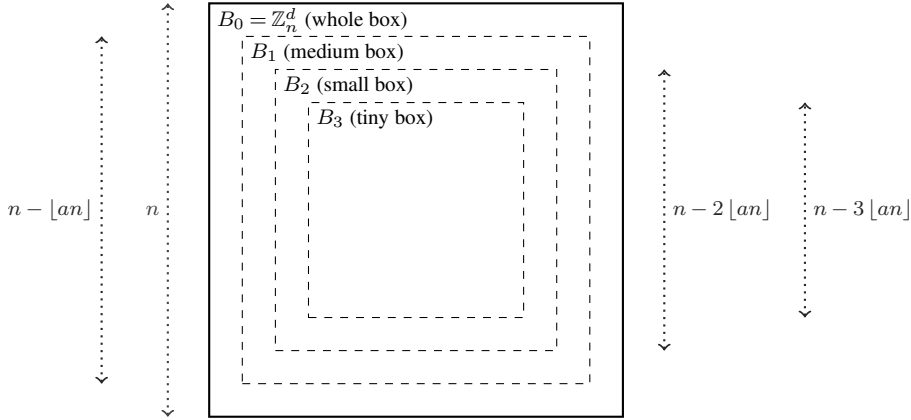
\begin{figure}
\begin{center}
\begin{tikzpicture}[scale=1.1]
\draw[thick] (-2.5,-2.5) -- (-2.5,2.5) -- (2.5,2.5) -- (2.5,-2.5) -- cycle;
\draw (-2.5,2.3) node[anchor=west]{$\boxO=\torus$ (whole box)};
\foreach \x in {2.1,1.7,1.3}
{
\draw[dashed] (-\x,-\x) -- (-\x,\x) -- (\x,\x) -- (\x,-\x) -- cycle;
}
\draw (-2.1,1.9) node[anchor=west]{$\boxI$ (medium box)};
\draw (-1.7,1.5) node[anchor=west]{$\boxII$ (small box)};
\draw (-1.3,1.1) node[anchor=west]{$\boxIII$ (tiny box)};
\draw[<->,dotted,thick,gray!50!black] (-3,-2.5) -- node[midway,left]{$n$} (-3,2.5);
\draw[<->,dotted,thick,gray!30!black] (-3.8,-2.1) -- node[midway,left]{$n-\Ent{an}$} (-3.8,2.1);
\draw[<->,dotted,thick,gray!30!black] (3,-1.7) -- node[midway,right]{$n-2\Ent{an}$} (3,1.7);
\draw[<->,dotted,thick,gray!30!black] (4.7,-1.3) -- node[midway,right]{$n-3\Ent{an}$} (4.7,1.3);
\end{tikzpicture}
\end{center}
\caption{\label{figBoxes}
The four sizes of boxes used for the proof of Theorem~\ref{thmLink}.
}
\end{figure}

Again, we use the site-wise representation of the model, and we study the number of topplings rather than the continuous time.
To stabilize the torus, we use the following toppling procedure (which is explained in more details in section~\ref{sectionProofThmLinkFinal}).
Denoting by~$p:\Z^d\to\torus$ the canonical projection application on the torus and writing, for every~$k\in\N$,
\begin{equation}
\label{defBoite}
\Lambda_k\ =\ \left(-\frac{k}{2},\,\frac{k}{2}\,\right]^d\cap\Z^d
\ =\ \left\{-\Ent{\frac{k-1}{2}},\,\ldots,\,\Ent{\frac{k}{2}}\right\}^d\,,
\end{equation}
we define four sizes of sub-boxes of the torus, as represented on figure~\ref{figBoxes}, by writing~\smash{$B_j=p(\Lambda_{n-j\Ent{an}})$} for every~\smash{$j\in\{0,\,1,\,2,\,3\}$}, with a certain parameter~$a>0$.
We call~$B_0=\torus$ the whole box,~$B_1$ the medium box,~$B_2$ the small box, and~$B_3$ the tiny box.

\begin{figure}
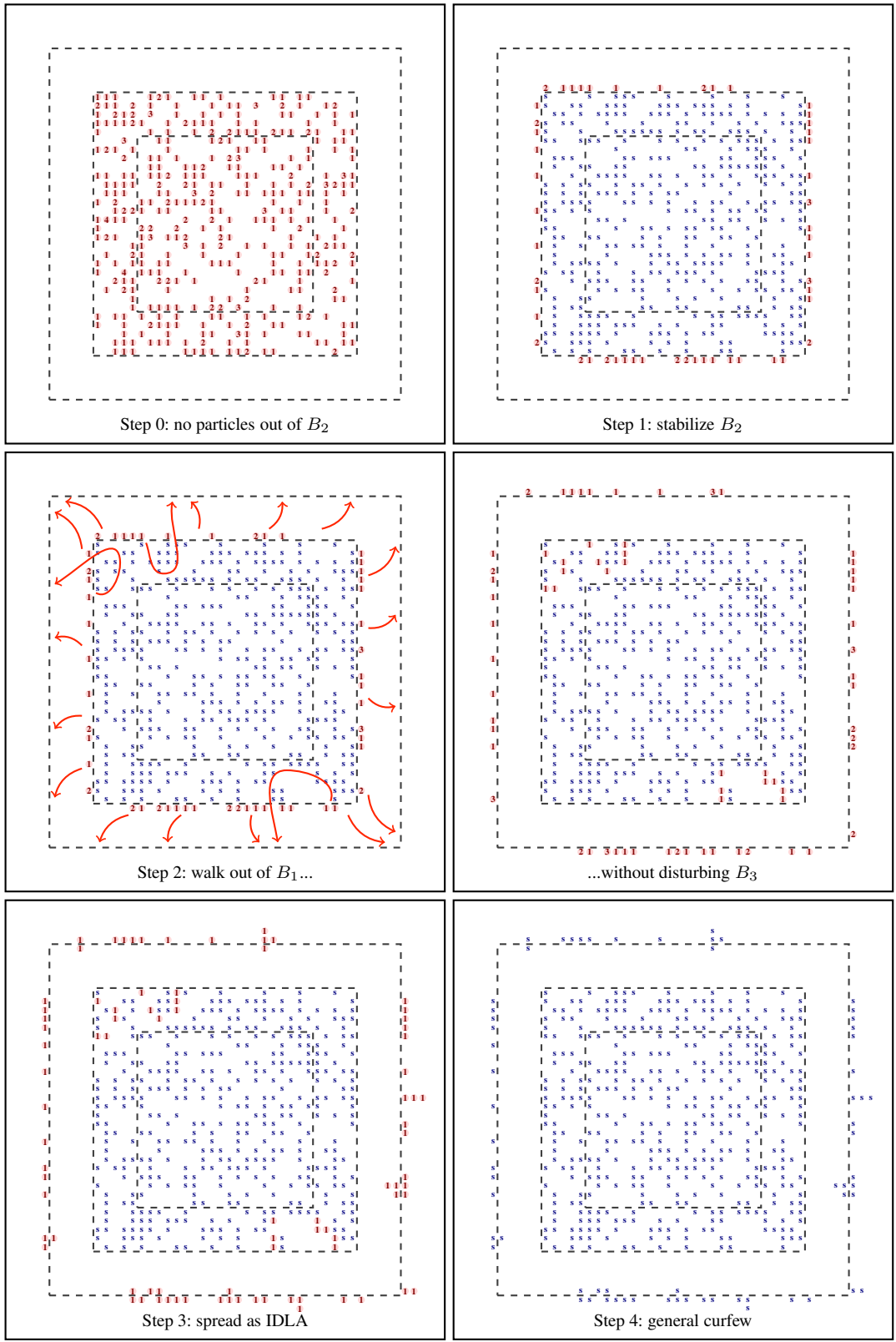

\begin{center}

\end{center}
\caption{\label{figTopplingThmLink}The toppling procedure used in the proof of Theorem~\ref{thmLink}.}
\end{figure}

If the parameter~$a$ is chosen small enough, then the volume of the remaining space between the small box and the boundary of the whole box is not too big, and thus we can assume with a \guillemets{reasonable} probabilistic cost that we start with no particles outside this small box.
What we call a \guillemets{reasonable} probabilistic cost is that of an event with possibly exponentially small probability, but much larger than~\smash{$e^{-cn^d}$} when~$c$ is fixed and~$a$ is chosen small enough (and~$n$ large enough).
For example the probability~$e^{-\mu|A|}$ of starting with an empty annulus~\smash{$A = B_0 \setminus B_2$} is exponentially small in~$n^d$, but for~$c > 0$ fixed it is larger than~\smash{$e^{-c n^d/24}$} for~$a$ small enough and~$n$ large enough.
Thus, to obtain the bound~(\ref{aimThm4}) we will write the event~\smash{$\{\calT_n<e^{cn^d}\}$} as an intersection of several events which have a \guillemets{reasonable} probabilistic cost.
Starting with no particles outside the small box is the first of these events, and we call it step~0 of the procedure.

Then, during step~1, we stabilize the small box, ignoring particles once they jump out of the box.
Because we assumed that~$\mu<\mu_c(\lambda)$, we have an upper bound on the number of particles which jump out of this box during this step, which is due to Rolla and Tournier (see Lemma~\ref{lemmaMn}): for every~$\varepsilon>0$, with positive probability, uniformly in~$n$, a proportion less than~$\varepsilon$ of the particles jump out of the small box.

After this, during step~2 we force these particles which jumped out of the small box to walk until they reach the boundary of the medium box, forbidding them to fall asleep before.
Intuitively, this can only increase the stabilization time, and this argument can be made rigorous using the notion of acceptable topplings (see section~\ref{sectionSitewise}).
By symmetry, the probability for a random walk starting from the border of the small box to reach the border of the medium box before reaching the tiny box is at least one half.
Hence, we can show that, with a \guillemets{reasonable} probabilistic cost, we can move all these particles to the boundary of the medium box, without any particle visiting the tiny box, where all the particles are sleeping.

Then, during step~3 we consider these particles which reached the boundary of the medium box during step~2 and we let each of them walk until it reaches an empty site.
Thus, these particles spread as an Internal Diffusion Limited Aggregation, which is a model corresponding to Activated Random Walks with an infinite sleep rate.
By using a much weaker estimate than Jerison, Levine and Sheffield's sharp bound to control its outer fluctuations (see section~\ref{sectionIDLA}), we can show that, with high probability, these particles can spread without hitting neither the small box nor the boundary of the whole box.

The fourth and last step consists in forcing all the active particles to fall asleep, which again, represents a \guillemets{reasonable} probabilistic cost.

If all the steps of this procedure succeed, we end up with a stable configuration, and the fact that no particle ever reaches the boundary of the whole box during this procedure allows us to upper bound the number of topplings performed (see section~\ref{sectionStabTime}), eventually yielding the property~(\ref{aimThm4}) about the stabilization time on the torus.

\subsection{Could our technique be extended for arbitrary sleep rates?}
\label{sectionEchelleOurs}

As explained above, we only prove that~$\mu_c(\lambda)<1$ for~$\lambda$ sufficiently small, and it is natural to wonder whether our technique could be improved to show that~$\mu_c(\lambda)<1$ for any finite value of the sleep rate~$\lambda$.

Looking into the various steps of our proof detailed in section~\ref{sectionSketchThmSlowPhase}, we can see that the most brutal approximation we make is when we only allow a particle to wake up the particle immediately preceding it.
Indeed, if~$\mu$ is very close to~$1$, then we can expect that most of the distances~$d(x_j,\,x_{j+1})$ will be of order~$1$, and thus when~$\lambda$ is large, the one-dimensional random walk will have a strong positive drift, making our technique fail.

If we allow every particle to wake up any other particle, and then topple the first awaken particle for the chosen order on the sites of~$A$, then the dynamics becomes much more complicated and our reduction to a simple one-dimensional process does not work anymore.

A softer approach could be to allow the particle on~$x_j$ to wake up either no particle, or the particle on~$x_{j-1}$, or the two particles on~$x_{j-1}$ and~$x_{j-2}$, etc.
Doing so, it would remain true that, after each step, the set of sleeping particles would remain of the form~$\{1,\,\ldots,\,j\}$, and thus we could control the procedure with a one-dimensional but long-range random walk.
However, our preliminary computations to control the drift of this random walk seem to indicate that this simple extension of the present strategy is not sufficient to prove~$\mu_c(\lambda)<1$ for all~$\lambda<\infty$.

\section{Diaconis-Fulton representation and properties}
\label{sectionDiaconis}

We now describe a site-wise construction known as the Diaconis-Fulton representation, which was introduced in~\cite{DF91} and adapted to the framework of Activated Random Walks in~\cite{RS12}.
This representation presents crucial properties of Abelianity and monotonicity, on which we rely heavily.
We essentially reproduce here notations from~\cite{Rolla20}.

\subsection{The site-wise construction}
\label{sectionSitewise}

Let~$G=(V,\,E)$ be a vertex-transitive and locally finite graph (in what follows, we will only be interested in~$G=\torus$ or~$G=\Z^d$).
A configuration of the Activated Random Walk model on~$G$ can be represented by a vector~$\eta:V\rightarrow\N_\s$, with~$\N_\s=\N\cup\{\s\}$, where~$\eta(x)=k\in\N$ means that there are~$k$ active particles on the site~$x$, while~$\eta(x)=\s$ means that there is one sleeping particle on the site~$x$.
We equip the set~$\N_\s$ with the total order~$0<\s<1<2<\ldots$ and, for~\smash{$\eta,\,\eta'\in\N_\s^V$}, we write~$\eta\leqslant\eta'$ whenever~$\eta(x)\leqslant\eta'(x)$ for every~$x\in V$.
If~\smash{$\eta\in\N_\s^V$} and~$A\subset V$, we introduce the following notation for the total number of particles on the sites of~$A$ in the configuration~$\eta$, regardless of their state, active or sleeping:
$$\abs{\eta}_A\ =\ \sum_{x\in A}\abs{\eta(x)}\,,$$
with the convention that~$\abs{\s}=1$.
If~$A=V$, we simply write~$\abs{\eta}=\abs{\eta}_V$.

We also define two operations on the configurations called toppling operations, which correspond to a particle moving or falling asleep.
A site~$x\in V$ is said to be stable in the configuration~$\eta$ if~$\eta(x)\in\{0,\,\s\}$, otherwise it is called unstable.
We say that a configuration~$\eta$ is stable in a subset~$U\subset V$ is~$\eta(x)\in\{0,\,\s\}$ for all~$x\in U$.
If~$\eta\in\N_\s^V$ and~$x\in V$ is unstable in~$\eta$, we define
\begin{equation}
\label{sleepToppling}
\tau_{x\s}\eta\ :\ z\in V\ \longmapsto\ \begin{cases}
\s &\text{if }z=x\text{ and }\eta(x)=1\,,\\
\eta(z) &\text{otherwise,}
\end{cases}
\end{equation}
and for every~$y\in V$ neighbour of~$x$ in the graph~$G$, we define
\begin{equation}
\label{jumpToppling}
\tau_{xy}\eta\ :\ z\in V\ \longmapsto\ \begin{cases}
\eta(x)-1 &\text{if }z=x\,,\\
\eta(y)+1 &\text{if }z=y\,,\\
\eta(z) &\text{otherwise,}
\end{cases}
\end{equation}
with the convention that~$\s+1=2$.
The key idea of the site-wise representation is that the toppling instructions are attached to the sites of the graph rather than to the particles.
Let us explain this.

We consider a field of instructions~$\tau=(\tau^{x,j})_{x\in V,\,j\in\N}$ where, for every~$x\in V$ and~$j\in\N$, the instruction~$\tau^{x,j}$ is either a sleep instruction~$\tau_{x\s}$ or a jump instruction~$\tau_{xy}$ to one of the neighbours~$y$ of~$x$ in the graph.
This field~$\tau$ will be taken random later on, but for the moment we consider a deterministic field of instructions.
We also consider a function~$h:V\rightarrow\N$ which we call odometer, and which will allow us to remember how many instructions have already been used at each site.
Given a field of instructions~$\tau=(\tau^{x,j})_{x\in V,\,j\in\N}$, a configuration~$\eta\in\N_\s^V$ and an odometer~$h\in\N^V$, if~$x\in V$ is unstable in~$\eta$ (i.e., if~$\eta(x)\geqslant 1$), we define the toppling operation at~$x$ as
\begin{equation}
\label{toppling}
\Phi^\tau_x(\eta,\,h)
\ =\ \big(\tau^{x,h(x)}\eta,\,h+\delta_x\big)\,,
\end{equation}
where~$\delta_x(y)=1$ if~$x=y$ and~$0$ otherwise.
We say that~$\Phi^\tau_x$ is legal for~$(\eta,\,h)$ if~$x$ is unstable in~$\eta$.

One can think of this representation as taking, above each site of the graph, an infinite pile of instructions which are either sleep instructions or jump instructions to a neighbouring site.
If, for every~$x\in V$,~$h(x)$ is the number of instructions already used at~$x$, then the toppling operation~$\Phi_x^\tau$ looks at the first instruction in the pile above~$x$ which has not already been used, and applies this instruction to the configuration~$\eta$.

In section~\ref{sectionProofThmLink} (where we prove Theorem~\ref{thmLink}), it will be convenient to be allowed to topple also sites where a particle is sleeping, which is not a legal toppling, but which we will call an acceptable toppling.
The above definitions~(\ref{jumpToppling}),~(\ref{sleepToppling}) and~(\ref{toppling}) extend to the case when~$\eta(x)=\s$ by letting~$\s-1=0$, and we say that~$\Phi^\tau_x$ is acceptable for~$(\eta,\,h)$ if~$\eta(x)\geqslant\s$.
One can think of an illegal but acceptable toppling as a forced awakening of the sleeping particle at~$x$, followed by a toppling of~$x$.

Given a sequence of sites~$\alpha=(x_1,\,\ldots,\,x_l)$, we say that~$\alpha$ is~$\tau$-legal (respectively,~$\tau$-acceptable) for~$(\eta,\,h)$ if, for every~$j\in\{1,\,\ldots,\,l\}$, the operation~$\Phi^\tau_{x_j}$ is legal (resp., acceptable) for~$\Phi^\tau_{x_{j-1}}\Phi^\tau_{x_{j-2}}\ldots\Phi^\tau_{x_1}(\eta,\,h)$ and, in this case, we write
$$\Phi^\tau_\alpha(\eta,\,h)
\ =\ \Phi^\tau_{x_{l}}\Phi^\tau_{x_{l-1}}\ldots\Phi^\tau_{x_2}\Phi_{x_1}^\tau(\eta,\,h)\,.$$
If~$\Phi^\tau_\alpha(\eta,\,0)=(\eta',\,h)$ where the resulting configuration~$\eta'$ is stable in~$U\subset V$, we say that the toppling sequence~$\alpha$ stabilizes~$\eta$ in~$U$.
For every sequence of sites~$\alpha=(x_1,\,\ldots,\,x_l)$, we define its odometer~$m_\alpha\in\N^V$ as
$$m_\alpha\ =\ \sum_{j=1}^l\delta_{x_j}\,,$$
which counts the number of times each site is toppled while applying~$\Phi^\tau_\alpha$.
For every sequence of sites~$\alpha=(x_1,\,\ldots,\,x_l)$, and every subset~$U\subset V$, we write~\smash{$\alpha\subset U$} if~$x_j\in U$ for every~$j\in\{1,\,\ldots,\,l\}$.
If~$\eta\in \N_\s^V$ is a configuration, we define the odometer of~$\eta$ in~$U$ as
\begin{equation}
\label{defOdometer}
m^\tau_{U,\,\eta}\ =\ \sup_{\alpha\subset U,\,\alpha\text{ legal}} m_\alpha\,.
\end{equation}
In the particular case~$U=V$, we omit~$U$ and simply write~$m^\tau_\eta$.
Let us also introduce the general notation
$$\norme{m}_A\ =\ \sum_{x\in A}m(x)$$
for every~\smash{$m\in\big(\N\cup\{+\infty\}\big)^V$}.

\subsection{Monotonicity properties}

A first notable property of the model is that adding particles to a configuration can only increase the number of topplings necessary to stabilize the configuration in a given set.
More precisely, we have (see Lemma~2.5 of~\cite{Rolla20}):

\begin{lemma}
\label{lemmaMonotonicity}
If~$\eta,\,\eta'\in\N_\s^V$ are such that~$\eta\leqslant\eta'$ and~$U\subset V$, then
$$m^\tau_{U,\,\eta}\ \leqslant\ m^\tau_{U,\,\eta'}\,.$$
\end{lemma}

This follows from the fact that for every toppling sequence~$\alpha\subset U$, if~$\alpha$ is~$\tau$-legal for~$\eta$ then it is also~$\tau$-legal for~$\eta'$.

We also have the following comparison property between acceptable and legal topplings (we reproduce here Lemma~2.1 of~\cite{Rolla20}):

\begin{lemma}
\label{lemmaAcceptable}
If~$\alpha$ is an acceptable sequence of topplings that stabilizes~$\eta$ in a finite subset~$U\subset V$,
and~$\beta\subset U$ is a legal sequence of topplings for~$\eta$, then~$m_\beta\leqslant m_\alpha$.
\end{lemma}

This Lemma allows us to consider acceptable sequences of topplings when we are looking for upper bounds on~$m_{U,\,\eta}^\tau$.
We will use it in section~\ref{sectionProofThmLink}.

\subsection{The Abelian property}
\label{sectionAbelian}

We now describe the Abelian property, which is a key feature of the Activated Random Walk model, which we use extensively.
Let~$U\subset V$ be a finite subset of the sites of the graph.
The following Lemma corresponds to Lemma~2.4 of~\cite{Rolla20}:

\begin{lemma}[Abelian property]
\label{lemmaAbelian}
If~$\alpha$ and~$\beta$ are both~$\tau$-legal toppling sequences for~$\eta$ that are contained in~$U$ and stabilize~$\eta$ in~$U$, then~$m_\alpha=m_\beta=m^\tau_{U,\,\eta}$ and~$\Phi^\tau_\alpha(\eta,\,0)=\Phi^\tau_\beta(\eta,\,0)$.
\end{lemma}

If~$A\subset U$, the Abelian property tells us that
$$\norme{m^\tau_{U,\,\eta}}_A
\ =\ \sum_{x\in A} m^\tau_{U,\,\eta}(x)$$
is the number of topplings on the sites of~$A$ necessary to stabilize the configuration~$\eta$ in~$U$ using the instructions in~$\tau$, regardless of the order with which the instructions are used (note that this number may be infinite).
In the particular case where~$A=U$, we simply write~\smash{$\|m^\tau_{U,\,\eta}\|=\|m^\tau_{U,\,\eta}\|_U$}.
If~\smash{$\|m^\tau_{U,\,\eta}\|<\infty$}, which means that~$\eta$ can be stabilized in~$U$ with a finite number of topplings, then the Abelian property ensures that the configuration reached when stabilizing~$\eta$ in~$U$ does not depend on the toppling sequence used, which allows us to write this final configuration~$\sigma^\tau_{U,\,\eta}$ (which is defined only if~\smash{$\|m^\tau_{U,\,\eta}\|<\infty$}).
In the case where~$V$ is finite and~$U=V$, we simply write~$\sigma^\tau_{\eta}=\sigma^\tau_{U,\,\eta}$.
With these notations, if~$V$ is finite, then for any~$\tau$-legal sequence~$\alpha$ which stabilizes~$\eta$, we have~\smash{$\Phi^\tau_\alpha(\eta,\,0)=(\sigma^\tau_\eta,\,m^\tau_{\eta})$}.

\subsection{Random instructions and dynamics}
\label{sectionRandomInstructions}

To recover the Activated Random Walk model described in paragraph~\ref{sectionDefARW}, one needs to draw a random field of instructions with a certain probability distribution.
Let~$\Pl$ be a probability measure on the set of all possible fields of instructions, which is such that the instructions~$(\tau^{x,j})_{x\in V,\,j\in\N}$ are independent, with, for every~$x\in V$ and~$j\in\N$,
\begin{equation}
\label{probaTau1}
\Pl\big(\,\tau^{x,j}=\tau_{x\s}\,\big)\ =\ \frac\lambda{1+\lambda}
\end{equation}
and, for every~$y$ neighbour of~$x$ in the graph,
\begin{equation}
\label{probaTau2}
\Pl\big(\,\tau^{x,j}=\tau_{xy}\,\big)\ =\ \frac 1{(1+\lambda)D}\,,
\end{equation}
where~$D$ is the common degree of all the vertices in the graph~$G$ (which is assumed to be vertex-transitive).
Given this random field~$\tau=(\tau^{x,j})_{x\in V,\,j\in\N}$ distributed according to~$\Pl$, the Activated Random Walk model is a process~$(\eta_t,\,h_t)_{t\geqslant 0}$, where~$h_0\equiv 0$ and the initial distribution~$\eta_0$ is independent of~$\tau$ and follows a translation-ergodic distribution on~$\N_\s^V$ with finite density, and with, for each site~$x\in V$, a transition rate of~$(1+\lambda)\abs{\eta_t(x)}\mathbf{1}_{\eta_t(x)\geqslant 1}$ from~$(\eta_t,\,h_t)$ to~$\Phi^\tau_x(\eta_t,\,h_t)$.
It is clear that such a process exists on any finite graph (like the torus~$\torus$), and we refer to~\cite{Rolla20} for a proof that such a process also exists on~$\Z^d$.

Thanks to Theorem~\ref{thmPhaseTransition}, to study the critical curve~$\mu_c(\lambda)$, we may restrict ourselves to the case where the initial distribution~$\eta_0$ is i.i.d.\ with a Poisson number of active particles with parameter~$\mu$ on each site (and no sleeping particles).
Let us denote by~$\Pm$ the corresponding probability measure on the set~$\N_s^V$.
We write~$\Pml=\Pm\otimes\Pl$ for the probability measure where the initial configuration~$\eta_0$ and the field of instructions~$\tau$ are independent and respectively distributed according to~$\Pm$ and~$\Pl$, and we write~$\Em$,~$\El$ and~$\Eml$ for the corresponding expectations.

In what follows, we will only work with the model on finite graphs, using Lemma~\ref{lemmaMn} below to relate the behaviour of the model on~$\Z^d$ to its behaviour in finite boxes.
If the graph~$V$ is finite, the Activated Random Walk model with Poisson i.i.d.\ initial distribution can be constructed as follows.
We draw the initial configuration~$\eta_0$ and the field of toppling instructions~$\tau$ distributed according to~$\Pml$, and we draw an infinite collection of Poisson point processes~$\theta_i\subset\R_+$ for~$i\in\N$, independent of~$(\eta_0,\,\tau)$, these processes being independent and having intensity~$1+\lambda$.
We draw an infinite number of clocks so that the collection~$(\theta_i)_{i\in\N}$ is independent of the number of particles in~$\eta_0$ but, in fact, we will only use the clocks~$\theta_i$ for~$i<|\eta_0|$.
The dynamics is then defined as follows.
At each time when one of these clocks rings, we choose a particle at random and, if this particle is active, we topple the site where it stands (in other words, we chose a site at random with the probability of choosing each site proportional to the number of particles on it, and we topple this site if it contains active particles).

\subsection{A property of the fixating phase}

To show that the slow phase on the torus~$\torus$ implies an active phase on the infinite lattice~$\Z^d$, we rely on the following result, which is a rephrasing of Theorem~2.11 of~\cite{Rolla20} which gives a sufficient condition for activity on~$\Z^d$ and which was originally proved by~\cite{RT18}.
Remember that~$\Lambda_n$ is the square box of~$\Z^d$ containing~$n^d$ sites and centered at~$0$, as defined by~(\ref{defBoite}).

\begin{lemma}
\label{lemmaMn}
Let us write
$$M_n\ =\ \abs{\eta_0}_{\Lambda_n}\,-\,\abs{\sigma_{\Lambda_n,\,\eta_0}^\tau}_{\Lambda_n}\,,$$
which is the number of particles which jump out of the box~$\Lambda_n\subset\Z^d$ when stabilizing the configuration~$\eta_0$ in~$\Lambda_n$ via legal topplings, with particles being killed upon leaving~$\Lambda_n$.
For every~$\lambda\in(0,\,\infty]$ and~$\mu<\mu_c(\lambda)$, we have
$$\lim\limits_{n\to\infty}\,\frac{\Eml\big[M_n\big]}{\abs{\Lambda_n}}\ =\ 0\,.$$
\end{lemma}

In other words, this Lemma tells us that if, when stabilizing an arbitrary large box, a uniformly positive fraction of particles jump out of the box with a uniformly positive probability, then~$\mu\geqslant\mu_c(\lambda)$.

\section{Proof of Theorem~\ref{thmSlowPhase}}
\label{sectionProofSlowPhase}

We now turn to the study of the fixation time on the torus~$\torus$, following the strategy presented in section~\ref{sectionSketchThmSlowPhase}.
We start with a series of preliminary estimates in sections~\ref{sectionSleepless} to~\ref{sectionMC}, before turning to the proof that the fixation time is exponentially large, first in terms of number of topplings from a deterministic initial configuration in section~\ref{sectionDeterministicEta}, then with a Poisson initial distribution in section~\ref{sectionSlowPhaseTopplings}, before the final proof of Theorem~\ref{thmSlowPhase} in section~\ref{sectionProofSlowPhaseFinal}.

\subsection{Preventing sleep instructions outside a given set}
\label{sectionSleepless}

As explained in paragraph~\ref{sectionSketchThmSlowPhase}, we first fix a set of sites~$A\subset \torus$, and we aim to find an upper bound on the probability that the system eventually fixates with one sleeping particle on each site of~$A$ (i.e., that it reaches the final stable configuration~$\sigma^\tau_\eta=\s\,\mathbf{1}_A$) within a given time or, rather, within a given number of topplings.
Since we are looking for a lower bound on this number of topplings, we may study the number of topplings on the sites of~$A$, which gives a lower bound on the total number of topplings.
Thus, our aim is to bound from above the probability
$$\Pl\Big(\,\big\{\norme{m^\tau_{\eta}}_A<M\big\}\cap\big\{\sigma^\tau_\eta=\s\mathbf{1}_A\big\}\,\Big)\,,$$
for a certain number~$M$ which will be chosen later, where we recall that~\smash{$\norme{m^\tau_{\eta}}_A$} denotes the number of toppling instructions performed on the sites of~$A$ during the complete stabilization of the configuration~$\eta$ in the torus using the instructions in~$\tau$ (do not get mixed with~\smash{$\norme{m_{A,\,\eta}^\tau}$}, which counts the number of topplings necessary to stabilize the configuration on~$A$, i.e., until there are no more active particles on the sites of~$A$).

The idea is now to modify the model by not allowing particles to sleep outside of~$A$ which, as we will show in Lemma~\ref{lemmaSleepless} below, can only increase the above probability.
Thus, we define a new distribution~$\PlA$ on the set of the fields of toppling instructions, which is such that we do not draw anymore sleep instructions on the sites outside of~$A$.
Namely, for every~$x\in A$ and~$j\in\N$, we have the same distribution for the instruction~$\tau^{x,j}$ given by~(\ref{probaTau1}) and~(\ref{probaTau2}), but whenever~$x\notin A$, we set
$$\PlA\big(\,\tau^{x,j}=\tau_{xy}\,\big)\ =\ \frac 1{2d}$$
for every~$y$ neighbour of~$x$ in~$\torus$, and~\smash{$\PlA(\tau^{x,j}=\tau_{x\s})=0$}, and different instructions remain of course independent.
We now prove the following:

\begin{lemma}
\label{lemmaSleepless}
For each initial configuration~$\eta:\torus\to\N_\s$, every~$A\subset \torus$ and every~$M\in\N$, we have the comparison
$$\Pl\Big(\,\big\{\norme{m^\tau_{\eta}}_A<M\big\}\cap\big\{\sigma^{\tau}_\eta=\s\,\mathbf{1}_A\big\}\,\Big)
\ \leqslant\ \PlA\Big(\,\norme{m^\tau_{\eta}}_A<M\,\Big)\,.$$
\end{lemma}

\begin{proof}
Let~$\eta:\torus\to\N_\s$, let~$A\subset \torus$ and~$M\in\N$.
Let us consider the natural coupling between~$\Pl$ and~$\PlA$ which consists in drawing~$\tau'$ according to~$\PlA$ and constructing~$\tau$ by inserting, between each jump instruction of~$\tau'$ on the sites~$x\notin A$, a number of sleep instructions which follows a geometric distribution with parameter~$\lambda/(1+\lambda)$, re-indexing the obtained stacks of instructions.

We now assume that~\smash{$\big\|m^\tau_{\eta}\big\|_A<M$} and~$\sigma^{\tau}_\eta=\s\,\mathbf{1}_A$, and we show that it implies that~\smash{$\big\|m^{\tau'}_{\eta}\big\|_A<M$}.
Since~$\sigma^{\tau}_\eta=\s\,\mathbf{1}_A$, we know that there exists a~$\tau$-legal toppling sequence~$\alpha=(x_1,\,\ldots,\,x_l)$ such that~$\Phi_\alpha^\tau(\eta,\,0)=(\s\,\mathbf{1}_A,\,m_\eta^\tau)$.
We now consider the toppling sequence~$\alpha'$ which is obtained from~$\alpha$ by removing the topplings which correspond to sleep instructions out of~$A$.

For each~$i\in\{0,\,\ldots,\,l\}$, let us write~$\alpha_i=(x_1,\,\ldots,\,x_i)$ (with the convention that~$\alpha_0$ is an empty toppling sequence), and let us write~$\alpha'_i$ for the sequence obtained from~$\alpha_i$ by removing the sleep instructions out of~$A$ (so that~$\alpha'_i$ may be shorter than~$\alpha_i$), and let us consider the configurations~$\eta_i$ and~$\eta'_i$ defined by~$\Phi_{\alpha_i}^\tau(\eta,\,0)=(\eta_i,\,m_{\alpha_i})$ and~\smash{$\Phi_{\alpha'_i}^{\tau'}(\eta,\,0)=(\eta'_i,\,m_{\alpha'_i})$}, if~\smash{$\Phi_{\alpha'_i}^{\tau'}$} is legal for~$(\eta,\,0)$.

Let us show by induction that, for every~$i\in\{0,\,\ldots,\,l\}$,
\begin{equation}
\label{inductionHypothesis}
\Phi_{\alpha'_i}^{\tau'}\text{ is legal for }(\eta,\,0)\,,\quad
\eta_i\ \leqslant\ \eta'_i
\quadet
\forall\,x\in A\quad\eta_i(x)\ =\ \eta'_i(x)\,.
\end{equation}
This is true for~$i=0$ since~$\eta_0=\eta'_0=\eta$.
Assume that~$i\in\{0,\,\ldots,\,l-1\}$ is such that the induction hypothesis~(\ref{inductionHypothesis}) holds.
Since~$x_{i+1}$ is by definition unstable in~$\eta_i$ (because~$\alpha$ is a~$\tau$-legal toppling sequence) and~$\eta'_i\geqslant\eta_i$, we know that~$x_{i+1}$ is also unstable in~$\eta'_i$, therefore~\smash{$\Phi_{\alpha'_{i+1}}^{\tau'}$} is legal for~$(\eta,\,0)$.

Let~$j=m_{\alpha_i}(x_{i+1})$.
If~$x_{i+1}\in A$ or~$\tau^{x_{i+1},j}$ is a jump instruction, then both~$\eta_{i+1}$ and~$\eta'_{i+1}$ are obtained from~$\eta_i$ and~$\eta'_i$ by performing the toppling operation~$\tau^{x_{i+1},j}$, thus~(\ref{inductionHypothesis}) remains satisfied at rank~$i+1$.
Otherwise, if~$x_{i+1}\notin A$ and~$\tau^{x_{i+1},j}$ is a sleep instruction, then we have~$\eta'_{i+1}=\eta'_i$ and~$\eta_{i+1}=\tau_{x_{i+1}\s}\eta_i$, whence~$\eta_{i+1}\leqslant\eta_i$ and~$\eta_{i+1}(x)=\eta_i(x)$ for all~$x\in A$.
Thus, we also still have~(\ref{inductionHypothesis}) at rank~$i+1$.

Therefore, it follows by induction that for all~$x\in A$,~$\eta'_l(x)=\eta_l(x)=\s$.
The number of particles being conserved, we deduce that~$\eta'_l=\s\,\mathbf{1}_A$, which means in particular that~$\alpha'$ is a~$\tau'$-legal toppling sequence which stabilizes~$\eta$.
By the Abelian property, this implies that~\smash{$m_\eta^{\tau'}=m_{\alpha'}$}.
Besides, by definition of~$\alpha'$, we have
$$\big\|m_{\alpha'}\big\|_A
\ =\ \big\|m_{\alpha}\big\|_A
\ =\ \big\|m^\tau_{\eta}\big\|_A
\ <\ M\,,$$
whence~\smash{$\big\|m^{\tau'}_{\eta}\big\|_A<M$}, which concludes the proof, since~$\tau'$ is distributed according to~$\PlA$ and~$\tau$ is distributed according to~$\Pl$.
\end{proof}

\subsection{Starting with one active particle on each sleeping site}
\label{sectionWalkToA}

We now show that, if we look for a lower bound on~\smash{$\big\|m^\tau_{\eta}\big\|_A$}, the worst initial configuration is the configuration~$\eta=\mathbf{1}_A$, where there is exactly one active particle on each site of~$A$.
Indeed, we prove:

\begin{lemma}
\label{lemmeWalkToA}
For any subset~$A\subset {\torus}$ and for any configuration~$\eta:\torus\to\N$ with~$|\eta|=|A|$ active particles, under the distribution~\smash{$\PlA$} defined in paragraph~\ref{sectionSleepless}, the number of topplings performed on~$A$ to stabilize the configuration~$\eta$ in~${\torus}$, namely~\smash{$\big\|m^\tau_{\eta}\big\|_A$}, stochastically dominates~\smash{$\big\|m^\tau_{\mathbf{1}_A}\big\|_A$}, the number of topplings performed on~$A$ to stabilize the configuration~$\mathbf{1}_A$ in~${\torus}$.
\end{lemma}

\begin{proof}
Let~$A\subset {\torus}$ and~$\eta:\torus\to\N$ such that~$|\eta|=|A|$.
We have to prove that, for all~$M\in\N$,
$$\PlA\Big(\,\norme{m^\tau_{\eta}}_A<M\,\Big)
\ \leqslant\ \PlA\Big(\,\norme{m^\tau_{\mathbf{1}_A}}_A<M\,\Big)\,.$$
Thus, we fix~$M\in\N$.
By the definition~(\ref{defOdometer}) of the odometer~$m_\eta^\tau$, we have (since we now deal with only one field of toppling instructions~$\tau$, we simply write \guillemets{legal} instead of~\guillemets{$\tau$-legal})
$$\norme{m^\tau_{\eta}}_A
\ =\ \sup\limits_{\alpha\text{ legal for }(\eta,\,0)}\ \norme{m_\alpha}_A\,,$$
whence
\begin{equation}
\label{ProbaForallLegal}
\PlA\Big(\,\norme{m^\tau_{\eta}}_A<M\,\Big)
\ =\ \PlA\Big(\,\forall\,\alpha\text{ legal for }(\eta,\,0)\,,\  \norme{m_\alpha}_A<M\,\Big)\,.
\end{equation}
Let~$\tau$ be distributed according to~$\PlA$, and let us consider the following toppling procedure: as long as there is at least an unstable site which is not in~$A$, or a site of~$A$ containing more than one particle, we topple one of these sites.
Note that, doing so, we never topple a site of~$A$ with a single particle, so that no particle can fall asleep during this procedure (remember that there are no sleep instructions out of~$A$ because~$\tau$ is drawn according to~$\PlA$).
We continue this procedure until there are no more sites of~$A$ with at least two particles and no more occupied site out of~$A$, and we denote by~$\calF$ the event that this procedure terminates in a finite number of topplings.
Since we took~$|\eta|=|A|$, we have~\smash{$\PlA(\calF)=1$}.
If this event~$\calF$ is realized, the procedure yields a finite legal sequence of topplings~$\beta$ such that~$\Phi^\tau_{\beta}(\eta,\,0)=(\mathbf{1}_A,\,m_{\beta})$, where~$\mathbf{1}_A$ is the configuration where there is exactly one particle on each site of~$A$.
In this case, for any sequence of topplings~$\gamma$ which is~$\tau$-legal for~$(\mathbf{1}_A,\,m_{\beta})$, the concatenation of~$\beta$ and~$\gamma$ yields a~$\tau$-legal sequence of topplings for~$(\eta,\,0)$, whence
\begin{align*}
\PlA\Big(\,\forall\,\alpha\text{ legal for }(\eta,\,0)\,,\  \norme{m_\alpha}_A<M\,\Big)
&\ \leqslant\ \PlA\Big(\,\forall\,\gamma\text{ legal for }(\mathbf{1}_A,\,m_{\beta})\,,\ \norme{m_{\beta}+m_\gamma}_A\,<\,M\ \Big|\ \calF\,\Big)\\
&\ \leqslant\ \PlA\Big(\,\forall\,\gamma\text{ legal for }(\mathbf{1}_A,\,m_{\beta})\,,\ \norme{m_\gamma}_A\,<\,M\ \Big|\ \calF\,\Big)\\
&\ =\ \PlA\Big(\,\forall\,\gamma\text{ legal for }(\mathbf{1}_A,\,0)\,,\ \norme{m_\gamma}_A\,<\,M\,\Big)\,.
\end{align*}
because, conditioned on~$\calF$, the remaining toppling instructions in~$\tau$ are still distributed according to~$\PlA$ and are independent of the instructions revealed to construct~$\beta$ (see Proposition~4 in~\cite{LS21} for a formal statement of this strong Markov property).
Combining this with equation~(\ref{ProbaForallLegal}), we obtain the claimed result.
\end{proof}

Therefore, we proved that the configuration~$\mathbf{1}_A$ is the worst initial configuration for the lower bound we are looking for.
Thus, in the sequel we forget about this initial configuration~$\eta$, and we reason as if we started with this new initial configuration~$\mathbf{1}_A$.

\subsection{Ordering the sleeping sites}
\label{sectionOrder}

We now start from this configuration~$\mathbf{1}_A$ with exactly one active particle on each site of~$A$, and we try to construct a random legal sequence of topplings~$\alpha$ such that~\smash{$\norme{m_\alpha}_A$} is exponentially large with overwhelming probability (under~$\PlA$).

Because particles are not allowed to fall asleep neither out of~$A$ nor on sites already occupied by another particle, we may consider a dynamics where, at each step, we choose a particle of~$A$, we topple it and, in case it chooses to move, we let it walk on~$\torus$ with successive topplings (in fact, we topple sites rather than particles) until it goes back to its starting point.
The sleeping particles met along this path are awaken by this particle.
Thus, the state space simplifies to~$\mathcal{P}(A)=\{S\subset A\}$, because, after each step (if we call step a loop from one site~$x\in A$ back to~$x$), each site of~$A$ is either occupied by a sleeping particle or by an active one.

To choose which particle we topple at each step, we define an ordering of the sites of~$A$.
The aim is that each site should not be too far apart from the preceding site.
We define recursively an ordering~$A=\{x_1,\,\ldots,\,x_k\}$.

First, we choose~$x_1\in A$ arbitrarily.
Then, if~$j\in\{1,\,\ldots,\,k-1\}$ is such that~$x_1,\,\ldots,\,x_j$ are constructed, we define the set~\smash{$A_j=A\setminus\{x_1,\,\ldots,\,x_j\}$} and we choose~$x_{j+1}\in A_j$ such that
$$d(x_j,\,x_{j+1})
\ =\ \min\limits_{x\in A_j}\,d(x_j,\,x)\,,$$
where~$d$ denotes the graph distance on~$\torus$.
In other words, at each step we choose among the remaining points the closest (or one of the closest) to the last point added.
For each~$j\in\{1,\,\ldots,\,k-1\}$, we write~\smash{$\ell_j=d(x_j,\,x_{j+1})$}.
Doing so, we obtain an ordering~$A=\{x_1,\,\ldots,\,x_k\}$ which is such that
\begin{equation}
\label{propertyOrdering}
\forall j\in\{1,\,\ldots,\,k-1\}\qquad
\ell_j\ =\ d(x_j,\,x_{j+1})\ =\ \min_{i>j}\,d(x_j,\,x_i)\,.
\end{equation}

\subsection{The toppling dynamic}
\label{sectionDynamic}

Using a toppling procedure based on the order on the sites of~$A$, we now prove a comparison between the Activated Random Walk model and a one-dimensional process.

\begin{lemma}
\label{lemmaDomStoch}
Consider~$A\subset\torus$, equipped with an order~$A=\{x_1,\,\ldots,\,x_k\}$ where~\smash{$k=|A|$}.
Let us consider a Markov chain with state space~$\{1,\,\ldots,\,k+1\}$ and transition probabilities given by
$$\left\{\begin{aligned}
&\ \forall j\in\{1,\,\ldots,\,k\}\qquad
p(j,\,j+1)\ =\ \frac{\lambda}{1+\lambda}\,,\\
&\ p(1,\,1)\ =\ \frac{1}{1+\lambda}\,,\qquad
p(k+1,\,k+1)\ =\ 1\,,\\
&\ \forall j\in\{2,\,\ldots,\,k\}\qquad
p(j,\,j-1)\ =\ \frac{1}{1+\lambda}P_{x_{j}}\big(T_{x_{j-1}}<T_{x_{j}}^+\big)\,,\\
&\ \forall j\in\{2,\,\ldots,\,k\}\qquad
p(j,\,j)\ =\ \frac{1}{1+\lambda}P_{x_{j}}\big(T_{x_{j}}^+<T_{x_{j-1}}\big)\,,
\end{aligned}\right.$$
where~$P_{x_{j}}$ denotes the probability with respect to the symmetric random walk on the torus~$\torus$ started from~$x_{j}$,~$T_{x_{j-1}}$ is the first hitting time of~$x_{j-1}$ and~$T_{x_{j}}^+$ is the first time the random walk comes back to its starting point~$x_{j}$.
We write~$\PCM_j$ for the probability measure of this Markov chain started from~$j$, and we denote by~$T_{k+1}$ its first hitting time of~$k+1$.
Then, the law of~\smash{$\big\|m^\tau_{\mathbb{1}_A}\big\|_A$} under~$\PlA$ stochastically dominates the law of~$T_{k+1}$ under~$\PCM_1$.
\end{lemma}

\begin{proof}
Let~$A=\{x_1,\,\ldots,\,x_k\}\subset\torus$, with~$k=|A|$.
Let us consider a random field of toppling instructions~$\tau$ distributed according to~$\PlA$.
We now specify a procedure which will almost surely yield a sequence of topplings~$\alpha$ legal for~$(\mathbf{1}_A,\,0)$, with a dynamics coupled with the aforementioned Markov chain in a way such that~$\norme{m_\alpha}_A\geqslant T_{k+1}$.

This procedure is defined as follows.
Remember that we start with exactly one active particle on each site of~$A$.
Our toppling procedure consists of a certain number of steps, indexed by~$t\in\N$, with a counter~$J(t)$ which indicates which site of~$A$ we have to topple at each step.

We start with~$J(0)=1$, and the procedure will terminate whenever we reach~$J(t)=k+1$.
We will ensure that, before each step~$t$, there is exactly one particle on each site of~$A$ and all the particles located on the sites~$x_i$ with~$i\geqslant J(t)$ are active (while the particles on the sites~$x_i$ with~$i< J(t)$ may be active or sleeping).

At each time step~$t$, if~$J(t)\leqslant k$, we topple the particle at site~$x_{J(t)}$ (see figure~\ref{figTopplingProcedure}).
With probability~$\lambda/(1+\lambda)$, this particle falls asleep before walking.
In this case, we set~$J(t+1)=J(t)+1$, and we proceed with the next step.
It remains true that, at time~$t+1$, all the particles on the sites~$x_i$ with~$i\geqslant J(t+1)$ are active.

Otherwise, with probability~$1/(1+\lambda)$, the particle at site~$x_{J(t)}$ chooses to move to one of the neighbouring sites on the torus.
We then let it walk with successive topplings until it comes back to its starting point~$x_{J(t)}$ (this is legal because it is not allowed to fall asleep outside of~$A$ and all the other sites of~$A$ are already occupied by other particles).
This loop almost surely terminates in a finite number of topplings, and then we leave the particle awake at site~$x_{J(t)}$.
If~$J(t)\geqslant 2$, along this path, with probability~\smash{$P_{x_{J(t)}}\big(T_{x_{J(t)-1}}<T_{x_{J(t)}}^+\big)$}, this active particle meets the particle at site~$x_{J(t)-1}$, which becomes active if it was sleeping (it may be sleeping or active before, but at least after this visit we know that it is active).
Thus, we set~$J(t+1)=J(t)-1$, meaning that, at the next step, we will topple the particle at site~$x_{J(t)-1}$.
This operation also ensures that at time~$t+1$, there are active particles on the sites~$x_i$ for~$i\geqslant J(t+1)$.

Note that the path may visit other sites~$x_i$ with~$i\notin\{J(t),\,J(t)-1\}$, but this does not matter, since the particles on the sites~$i> J(t)$ are already awake and we do not care about the state of the particles on the sites~$i<J(t)-1$.
As represented on figure~\ref{figTopplingProcedure}, some of the latter particles may be active and some may be sleeping.

During this procedure, the process~$J(t)$ evolves exactly like the Markov chain described above (at least almost surely, because the procedure can also fail with probability~$0$ if at some step the random walk on the torus never comes back to its starting point).
Thus, the number~$T$ of steps before the procedure stops is distributed as the reaching time~$T_{k+1}$ under~$\PCM_1$.
What's more, at each step of the procedure, we topple at least one site of~$A$.
Therefore, the procedure almost surely yields a finite sequence of topplings~$\alpha$ which is legal for~$(\mathbf{1}_A,\,0)$ and which is such that~\smash{$\big\|m_\alpha\big\|_A\geqslant T$}, which implies that~\smash{$\big\|m_{\mathbf{1}_A}^\tau\big\|_A\geqslant T$}.
Therefore, we constructed a coupling between~$\PlA$ and~$\PCM_1$ which is such that~\smash{$\big\|m_{\mathbf{1}_A}^\tau\big\|_A\geqslant T_{k+1}$} almost surely, which proves the claimed stochastic domination.
\end{proof}

\subsection{An estimate about random walks on the torus}
\label{sectionRWtorus}

We now prove the following geometric estimate about random walks on the torus:

\begin{lemma}
\label{lemmaRWtorus}
Denoting by~$P_x$ the probability measure relative to the symmetric random walk on the torus~$\torus=(\Z/n\Z)^d$ started at~$x$, writing~$T_y$ for its first hitting time of~$y$ and~$T_x^+$ its first return time to~$x$ and~$d$ for the graph distance on~$\torus$, we have
$$\forall\,n\geqslant 1\quad
\forall\,x,\,y\in\torus\qquad
x\neq y\quadimplique
P_x\big(T_y<T_x^+\big)\ \geqslant\ \frac{1}{(2d)\,d(x,\,y)}\,.$$
\end{lemma}

\begin{proof}
The result is a straightforward consequence of the theory connecting random walks and electric networks (see for example~\cite{DS84} or~\cite{LP16}).
Indeed, the above probability is related to the effective resistance~$R(x,\,y)$ between~$x$ and~$y$ on the torus where each edge has a resistance equal to~$2d$, through the relation
$$R(x,\,y)\ =\ \frac{1}{P_x\big(T_y<T_x^+\big)}\,.$$
According to Rayleigh's monotonicity law, this resistance can only increase if we remove edges from the graph.
Thus, removing all the edges except the ones belonging to one of the shortest paths from~$x$ to~$y$, we obtain that
$$R(x,\,y)\ \leqslant\ 2d\,d(x,\,y)\,,$$
whence the claimed inequality.
\end{proof}

\subsection{A geometric estimate}
\label{sectionSommeLog}

We now prove the following property of the ordering that we defined in paragraph~\ref{sectionOrder}:

\begin{lemma}
\label{lemmaSommeLog}
There exists a universal constant~$\kappa_d>0$, which depends only on the dimension~$d$, such that, for all~$n\geqslant 1$ and for every subset~$A\subset\torus$, the distances~$\ell_1,\,\ldots,\,\ell_{|A|-1}$ defined in section~\ref{sectionOrder} satisfy
$$\prod_{j=1}^{|A|-1}\ell_j
\ \leqslant\ (\kappa_d)^{n^d}\,.$$
\end{lemma}

\begin{proof}
Let~$A\subset\torus$.
We consider the order~$A=\{x_1,\,\ldots,\,x_k\}$ with~$k=|A|$ and the distances~$\ell_1,\,\ldots,\,\ell_{k-1}$, as defined in paragraph~\ref{sectionOrder}.
For every~$L\in\N\setminus\{0\}$, let us consider the set
$$\calA_L\ =\ \Big\{\,j\in\{1,\,\ldots,\,k-1\}\ :\ \ell_j\geqslant L\,\Big\}\,.$$
If~$L>dn/2$, then we have~$\calA_L=\varnothing$, because there is no pair of points in~$\torus$ separated by a distance~$L$.
Hence, we assume that~$L\leqslant dn/2$.
This set~$\calA_L$ has the property that any two different points with indices in~$\calA_L$ are always separated by a distance at least~$L$, because if~$i,\,j\in\calA_L$ with for example~$i>j$, then the property~(\ref{propertyOrdering}) of the ordering ensures that~$d(x_i,x_j)\geqslant \ell_j\geqslant L$.
Defining the open balls for the graph distance~$d$ on the torus~$\torus$ by
$$\forall\,x\in\torus\quad
\forall\,r\in\N\qquad
B_{\torus}(x,r)\ =\ \big\{\,y\in\torus\,:\,d(x,y)<r\,\big\}\,,$$
we know that the open balls~$B_{\torus}(x,\,\Ceil{L/2})$ for~$x\in\calA_L$ are disjoint, whence
\begin{equation}
\label{disjointBalls}
\abs{\calA_L}\ \leqslant\ \frac{\abs{\torus}}{\abs{B_{\torus}(0,\,\Ceil{L/2})}}
\ \leqslant\ \frac{\abs{\torus}}{\abs{B_{\torus}(0,\,\Ceil{L/(2d)})}}\,.
\end{equation}
Now note that, since~$L/(2d)\leqslant n/4$, we have
$$\abs{B_{\torus}\left(0,\,\Ceil{\frac{L}{2d}}\right)}
\ =\ \abs{B_1\left(0,\,\Ceil{\frac{L}{2d}}\right)}
\ \geqslant\ \abs{B_\infty\left(0,\,\Ceil{\frac{\Ceil{L/(2d)}}{d}}\right)}\,,$$
where~$B_1$ denotes the open ball on~$\Z^d$ for the graph distance on~$\Z^d$, which corresponds to the usual~$L_1$-norm~$\norme{\cdot}_1$, while~$B_\infty$ denotes the open ball on~$\Z^d$ for the supremum norm~$\norme{\cdot}_\infty$.
Thus, we can write
$$\abs{B_{\torus}\left(0,\,\Ceil{\frac{L}{2d}}\right)}
\ \geqslant\ \abs{B_\infty\left(0,\,\Ceil{\frac{L}{2d^2}}\right)}
\ =\ \left(2\Ceil{\frac{L}{2d^2}}-1\right)^d\\
\ \geqslant\ \left(\frac L {2d^2}\right)^d\,.$$
Plugging this into~(\ref{disjointBalls}), we obtain, for every~$L\leqslant dn/2$,
$$\abs{\calA_L}
\ \leqslant\ \left(\frac{2d^2n}{L}\right)^d\,.$$
The above is also true when~$L>dn/2$, since in this case we noted that~$\calA_L=\varnothing$.
Therefore, we may write
$$\sum_{j=1}^{k-1}\ln(\ell_j)
\ \leqslant\ \sum_{m=0}^{+\infty}\big(\abs{\calA_{2^m}}-\abs{\calA_{2^{m+1}}}\big)\ln(2^{m+1})
\ =\ \sum_{m=0}^{+\infty}\abs{\calA_{2^m}}\ln 2
\ \leqslant\ n^d\ln \kappa_d\,,$$
where the constant~$\kappa_d$ is given by
$$\kappa_d\ =\ 2\exp\Bigg(\,\big(2d^2\big)^d\sum_{m=0}^{+\infty}\frac{m+1}{2^{md}}\,\Bigg)
\ =\ 2\exp\Bigg(\,\frac{\big(2d^2\big)^d}{1-2^{-d}}\,\Bigg)\,,$$
which is indeed finite for any~$d\geqslant 1$.
\end{proof}

\subsection{Study of the absorption time of the one-dimensional Markov chain}
\label{sectionMC}

We now study the law of the reaching time~$T_{k+1}$ of the absorbing state~$k+1$ for the Markov chain that we defined in paragraph~\ref{sectionDynamic}, and we prove:

\begin{lemma}
\label{lemmaMC}
Let~$\lambda>0$,~$\mu\in(0,\,1)$ and~$A\subset\torus$ with~$|A|=\Ceil{\mu n^d}$.
We consider the order~$A=\{x_1,\,\ldots,\,x_k\}$ defined in paragraph~\ref{sectionOrder}, with~$k=|A|$, and we consider the Markov chain described in Lemma~\ref{lemmaDomStoch}, with law~$\PCM_1$.
Then, for any~$M\in\N$, the absorption time~$T_{k+1}$ of this Markov chain satisfies
$$\PCM_1\big(T_{k+1}\leqslant M\big)
\ \leqslant\ \frac{M}{1\wedge(2d\lambda)}\big(\kappa_d(2d\lambda)^\mu\big)^{n^d}\,.$$
\end{lemma}

\begin{proof}
Let~$\lambda>0$,~$\mu\in(0,\,1)$,~$A=\{x_1,\,\ldots,\,x_k\}\subset\torus$ with~$|A|=k=\Ceil{\mu n^d}$, as constructed in paragraph~\ref{sectionOrder}, and let~$M\in\N$.
We can modify the Markov chain defined in Lemma~\ref{lemmaDomStoch} by letting
$$p(k+1,\,k)
\ =\ p(k,\,k+1)
\ =\ \frac{\lambda}{1+\lambda}
\qquadet
p(k+1,\,k+1)\ =\ \frac{1}{1+\lambda}\,,$$
since this does not change the reaching time of~$k+1$, which we are interested in.
With this modification, the Markov chain becomes reversible with respect to the measure~$\nu$ defined on~$\{1,\,\ldots,\,k+1\}$ by
$$\forall j\in\{1,\,\ldots,\,k+1\}\qquad
\nu(j)\ =\ \prod_{1\leqslant i<j}\frac{p(i,\,i+1)}{p(i+1,\,i)}\,,$$
with the convention that~$\nu(1)=1$.
This reversibility allows us to write
\begin{align*}
\PCM_1(T_{k+1}\leqslant M)
&\ =\ \sum_{t=1}^M\PCM_1(T_{k+1}=t)
\ =\ \sum_{t=1}^M\ \sum_{1=j_0,\,j_1,\,\ldots,\,j_t=k+1}\ \prod_{i=1}^t p(j_{i-1},\,j_i)\\
&\ =\ \sum_{t=1}^M\,\sum_{1=j_0,\,j_1,\,\ldots,\,j_{t}=k+1}\ \prod_{i=1}^{t} \frac{\nu(j_i)}{\nu(j_{i-1})}\,p(j_i,\,j_{i-1})
\ =\ \sum_{t=1}^M\,\nu(k+1)\ \PCM_{k+1}\big(J(t)=1\big)\\
&\ \leqslant\ \sum_{t=1}^M\,\nu(k+1)
\ =\ M\,\nu(k+1)\,.\numberthis\label{calculRev}
\end{align*}
We now study~$\nu(k+1)$ to show that it is exponentially small.
We have
$$\nu(k+1)\ =\ \prod_{j=1}^{k}\frac{p(j,\,j+1)}{p(j+1,\,j)}
\ =\ \prod_{j=1}^{k-1}\frac{p(j,\,j+1)}{p(j+1,\,j)}
\ =\ \lambda^{k-1}\,\prod_{j=1}^{k-1} \frac 1 {P_{x_{j+1}}\big(T_{x_j}<T_{x_{j+1}}^+\big)}\,.$$
Using the results of Lemmas~\ref{lemmaRWtorus} and~\ref{lemmaSommeLog}, this becomes
$$\nu(k+1)
\ \leqslant\ (2d\lambda)^{k-1}\,\prod_{j=1}^{k-1}\ell_j
\ \leqslant\ (2d\lambda)^{k-1} (\kappa_d)^{n^d}\,.$$
Plugging this into~(\ref{calculRev}), we obtain
$$\PCM_1(T_{k+1}\leqslant M)
\ \leqslant\ M(2d\lambda)^{k-1} (\kappa_d)^{n^d}\,.$$
Recalling that~$k=\Ceil{\mu n^d}$, we have
$$(2d\lambda)^{k-1}
\ \leqslant\ \frac{(2d\lambda)^{\mu n^d}}{1\wedge(2d\lambda)}\,,$$
whence the announced result.
\end{proof}

\subsection{Exponential number of topplings for a deterministic initial configuration}
\label{sectionDeterministicEta}

We now show how to combine the above results to prove that, still on the torus~$\torus$, starting with a deterministic initial configuration with a prescribed density of active particles, the number of topplings necessary to reach a stable configuration is exponentially large with big probability, as stated in the following Lemma:

\begin{lemma}
\label{lemmaExpTopplingsDeter}
In any dimension~$d\geqslant 1$, for every~$\mu\in(0,1)$ and~$\lambda>0$ satisfying the condition~(\ref{conditionMuLambda}), there exists~$c>0$ such that, for~$n$ large enough, we have
$$\sup_{\eta\in\N^{\torus}\,:\,\abs{\eta}\geqslant\mu n^d}
\ \Pl\Big(\,\norme{m_\eta^\tau}\,<\,e^{cn^d}\,\Big)
\ <\ e^{-cn^d}\,.$$
\end{lemma}

\begin{proof}
Let~$\lambda\in(0,\,\infty)$ and~$\mu\in(0,\,1)$ which satisfy the condition~(\ref{conditionMuLambda}).
This condition being a strict inequality, we may take~$c>0$ such that
\begin{equation}
\label{conditionWithC}
\frac{\kappa_d(2d\lambda)^\mu}{\mu^\mu (1-\mu)^{1-\mu}}
\ <\ e^{-2c}\,.
\end{equation}
Let us start with a configuration~\smash{$\eta\in\N^{\torus}$} such that~\smash{$|\eta|= k = \Ceil{\mu n^d}$}.
Let us recall that~$\|m_\eta^\tau\|$ denotes the total number of topplings necessary to stabilize the configuration~$\eta$ using the toppling instructions in~$\tau$.
For any~$M\in\N$, using Lemmas~\ref{lemmaSleepless},~\ref{lemmeWalkToA},~\ref{lemmaDomStoch} and~\ref{lemmaMC}, we can write
\begin{align*}
\Pl\Big(\,\norme{m_\eta^\tau}\leqslant M\,\Big)
&\ \stackrel{\phantom{\text{(Lemma~\ref{lemmaSleepless})}}}{=}\ \sum_{|A|=k}\Pl\Big(\,\big\{\norme{m_\eta^\tau}\leqslant M\big\}\cap\big\{\sigma^\tau_\eta=\s\mathbf{1}_A\big\}\,\Big)\\
&\ \stackrel{\phantom{\text{(Lemma~\ref{lemmaSleepless})}}}{\leqslant}\ \sum_{|A|=k}\Pl\Big(\,\big\{\norme{m^\tau_{\eta}}_A\leqslant M\big\}\cap\big\{\sigma^\tau_\eta=\s\mathbf{1}_A\big\}\,\Big)\\
&\ \stackrel{\text{(Lemma~\ref{lemmaSleepless})}}{\leqslant}\ \sum_{|A|=k}\PlA\Big(\,\norme{m^\tau_{\eta}}_A\leqslant M\,\Big)\\
&\ \stackrel{\text{(Lemma~\ref{lemmeWalkToA})}}{\leqslant}\ \sum_{|A|=k}\PlA\Big(\,\norme{m_{\mathbf{1}_A}^\tau}_A\leqslant M\,\Big)\\
&\ \stackrel{\text{(Lemma~\ref{lemmaDomStoch})}}{\leqslant}\ \sum_{|A|=k} \PCM_1\big(T_{k+1}\leqslant M\big)\\
&\ \stackrel{\text{(Lemma~\ref{lemmaMC})}}{\leqslant}\ \binom{n^d}{\Ceil{\mu n^d}} \frac{M}{1\wedge(2d\lambda)} \big(\kappa_d(2d\lambda)^\mu\big)^{n^d}\,.\numberthis\label{combiningLemmas}
\end{align*}
Now note that Stirling's formula ensures that, when~$m\to+\infty$,
\begin{align*}
\binom{m}{\Ceil{\mu m}}
&\ =\ \frac{m!}{\Ceil{\mu m}!(m-\Ceil{\mu m})!}
\ \stackrel{m\to\infty}{\sim}\ \frac{m^m\sqrt{2\pi m}}{\Ceil{\mu m}^{\Ceil{\mu m}}\sqrt{2\pi\mu m}(m-\Ceil{\mu m})^{m-\Ceil{\mu m}}\sqrt{2\pi(1-\mu)m}}\\
&\ =\ O\left(\frac{m^m}{(\mu m)^{\mu m}(m-\mu m)^{m-\mu m}\sqrt{m}}\right)
\ =\ O\left(\frac{1}{\mu^{\mu m}(1-\mu)^{(1-\mu)m}\sqrt{m}}\right)\,.
\end{align*}
Thus, there exists~$C=C(\mu)<\infty$ such that, for all~$m\geqslant 1$,
$$\binom{m}{\Ceil{\mu m}}
\ \leqslant\ \frac{C}{\mu^{\mu m}(1-\mu)^{(1-\mu)m}\sqrt{m}}\,.$$
Plugging this into~(\ref{combiningLemmas}) and taking~\smash{$M=\big\lfloor e^{cn^d}\big\rfloor$}, we obtain
$$\Pl\Big(\,\norme{m_\eta^\tau}<e^{cn^d}\,\Big)
\ \leqslant\ \frac{C}{\big(1\wedge(2d\lambda)\big)n^{d/2}}\Bigg[\,\frac{\kappa_d(2d\lambda)^\mu e^{c}}{\mu^\mu(1-\mu)^{1-\mu}}\,\Bigg]^{n^d}\,.$$
With our assumption~(\ref{conditionWithC}) on the constant~$c$, this becomes
$$\Pl\Big(\,\norme{m_\eta^\tau}<e^{cn^d}\,\Big)
\ \leqslant\ \frac{C}{\big(1\wedge(2d\lambda)\big)n^{d/2}}\,e^{-cn^d}\,.$$
Since this constant~$c$ does not depend on the initial configuration~$\eta$, we deduce that, for~$n$ large enough,
$$\sup_{\eta\in\N^{\torus}\,:\,\abs{\eta}=\Ceil{\mu n^d}}
\ \Pl\Big(\,\norme{m_\eta^\tau}\,<\,e^{cn^d}\,\Big)
\ <\ e^{-cn^d}\,.$$
By virtue of the monotonicity property stated in Lemma~\ref{lemmaMonotonicity}, this estimate remains true if the supremum is taken over all the configurations~$\eta$ with~$|\eta|\geqslant\mu n^d$.
\end{proof}

\subsection{Exponential number of topplings with a Poisson initial distribution}
\label{sectionSlowPhaseTopplings}

The result of Lemma~\ref{lemmaExpTopplingsDeter} easily extends to an initial configuration with Poisson i.i.d.\ numbers of active particles on each site of the torus, that is to say, distributed according to~$\Pm$.

\begin{lemma}
\label{lemmaSlowPhaseTopplings}
In any dimension~$d\geqslant 1$, for every~$\mu\in(0,1)$ and~$\lambda>0$ satisfying the condition~(\ref{conditionMuLambda}), there exists~$c>0$ such that, for~$n$ large enough, we have
$$\Pml\Big(\,\norme{m_{\eta_0}^\tau}\,<\,e^{cn^d}\,\Big)
\ <\ e^{-cn^d}\,.$$
\end{lemma}

\begin{proof}
Let~$\lambda\in(0,\,\infty)$ and~$\mu\in(0,\,1)$ satisfying the condition~(\ref{conditionMuLambda}).
Because this condition is a strict inequality, we may fix~$0<\mu'<\mu$ such that the condition~(\ref{conditionMuLambda}) still holds for~$\lambda$ and~$\mu'$.
Let us take~$c>0$, which will be chosen later.

Remember that~$\Pml$ is a joint probability measure where the initial configuration~$\eta_0$ is distributed according to~$\Pm$, i.e.,\ with i.i.d.\ Poisson numbers of active particles on each site, and the field of toppling instructions~$\tau$ is distributed according to~$\Pl$,~$\eta_0$ and~$\tau$ being independent.
By conditioning on~$\eta_0$, we may write
\begin{align*}
\Pml\Big(\,\norme{m_{\eta_0}^\tau}\,<\,e^{cn^d}\,\Big)
&\ =\ \sum_{\eta\in\N^{\torus}}\,
\Pm\big(\eta_0=\eta\big)
\ \Pl\Big(\,\norme{m_\eta^\tau}\,<\,e^{cn^d}\,\Big)\\
&\ \leqslant\ \Pm\Big(\,|\eta_0|<\mu'n^d\,\Big)
\,+\,\sup_{\eta\in\N^{\torus}\,:\,\abs{\eta}\geqslant\mu' n^d}
\Pl\Big(\,\norme{m_\eta^\tau}\,<\,e^{cn^d}\,\Big)\,.\numberthis\label{conditioningEta}
\end{align*}
Remember that, under~$\Pm$, the initial configuration~$\eta_0$ contains a total number of active particles~$|\eta_0|$ which follows a Poisson distribution with parameter~$\mu n^d$, whence, using again Stirling's formula,
\begin{multline*}
\Pm\Big(\,|\eta_0|<\mu'n^d\,\Big)
\ =\ \sum_{j<\mu'n^d}\frac{(\mu n^d)^j}{j!}e^{-\mu n^d}
\ \leqslant\ e^{-\mu n^d}\Ceil{\mu' n^d}\frac{(\mu n^d)^{\mu' n^d}}{(\mu' n^d)!}\\
\ \eqninfty\ e^{-\mu n^d}(\mu' n^d)\frac{(\mu n^d)^{\mu' n^d}}{(\mu' n^d)^{\mu' n^d}e^{-\mu' n^d}\sqrt{2\pi\mu' n^d}}
\ =\ \sqrt{\frac{\mu' n^d}{2\pi}}\,e^{-c_1n^d}
\quadavec
c_1\ =\-(\mu-\mu')-\mu'\ln\left(\frac{\mu}{\mu'}\right)\,.
\end{multline*}
Note that since we took~$\mu'<\mu$, the convexity of the logarithm function ensures that~\smash{$c_1>0$}.
Thus, for~$n$ large enough, we have
\begin{equation}
\label{PoissonInitialNumber}
\Pm\Big(\,|\eta_0|<\mu'n^d\,\Big)
\ \leqslant\ e^{-c_1 n^d/2}\,.
\end{equation}
Besides, let us take~$c_2>0$ given by Lemma~\ref{lemmaExpTopplingsDeter}, applied with~$\lambda$ and~$\mu'$, so that
\begin{equation}
\label{ApplicationLemmaExpTopplingsDeter}
\sup_{\eta\in\N^{\torus}\,:\,\abs{\eta}\geqslant\mu' n^d}
\Pl\Big(\,\norme{m_\eta^\tau}\,<\,e^{c_2n^d}\,\Big)
\ \leqslant\ e^{-c_2 n^d}\,.
\end{equation}
If we choose~$c>0$ such that~$c<\min(c_1/2,\,c_2)$, then plugging~(\ref{PoissonInitialNumber}) and~(\ref{ApplicationLemmaExpTopplingsDeter}) into~(\ref{conditioningEta}), we get
$$\Pml\Big(\,\norme{m_{\eta_0}^\tau}\,<\,e^{cn^d}\,\Big)
\ \leqslant\ e^{-c_1 n^d}+e^{-c_2 n^d}
\ <\ e^{-cn^d}\,,$$
for~$n$ large enough.
\end{proof}

\subsection{Concluding proof of Theorem~\ref{thmSlowPhase}}
\label{sectionProofSlowPhaseFinal}

We are now in a position to conclude the proof of Theorem~\ref{thmSlowPhase}, which easily follows from Lemma~\ref{lemmaSlowPhaseTopplings}.

\begin{proof}[Proof of Theorem~\ref{thmSlowPhase}]
There only remains to deal with the Poisson clocks, to translate the result of Lemma~\ref{lemmaSlowPhaseTopplings} about the number of topplings into a result on the fixation time of the continuous-time process.

As explained in paragraph~\ref{sectionRandomInstructions}, the Activated Random Walk model on the torus can be constructed by drawing the initial configuration~$\eta_0$ and the field of toppling instructions~$\tau$ distributed according to~$\Pml$, and then a collection of independent Poisson process~$(\theta_i)_{i\in\N}$ of intensity~$1+\lambda$, independent of~$\eta_0$ and~$\tau$.
Each time one of the clocks~$(\theta_i)_{i<|\eta_0|}$ rings, we choose a particle (not a site) uniformly at random, and topple its site if this particle is active at that moment (and we do nothing if the chosen particle is inactive).
With this construction, we can write, for~$c>0$ which will be chosen later,
\begin{multline}
\label{conditionningPoisson}
\PARW\Big(\,\calT_n\,<\,e^{cn^d}\,\Big)
\ \leqslant\ \PARW\Big(\,\big\{\calT_n\,<\,e^{cn^d}\big\}\cap\big\{\abs{\eta_0}\leqslant n^d\big\}\cap\big\{\norme{m_{\eta_0}^\tau}\geqslant e^{3cn^d}\big\}\,\Big)\\
+\,\Pm\Big(\,\abs{\eta_0}>n^d\,\Big)
\,+\,\Pml\Big(\,\norme{m_{\eta_0}^\tau}\,<\,e^{3cn^d}\,\Big)\,.
\end{multline}
We know that the clocks~$(\theta_i)_{i<|\eta_0|}$ ring a total number of times at least~$\|m_{\eta_0}^\tau\|$ during the time interval~$[0,\,\calT_n]$, whence
\begin{align*}
\PARW\Big(\,\big\{\calT_n\,<\,e^{cn^d}\big\}\cap\big\{\abs{\eta_0}\leqslant n^d\big\}\cap\big\{\norme{m_{\eta_0}^\tau}\geqslant e^{3cn^d}\big\}\,\Big)
&\ \leqslant\ 
\PARW\bigg(\,\sum_{i<n^d}\Big|\theta_i\cap\big[0,\,e^{cn^d}\big]\Big|\,\geqslant\,e^{3cn^d}\,\bigg)\\
&\ \leqslant\ n^d\,\mathbb{E}_\mu^\lambda\Big[\big|\theta_0\cap\big[0,\,e^{cn^d}\big]\big|\Big]e^{-3cn^d}
\ =\ n^d (1+\lambda) e^{-2cn^d}\,,\numberthis\label{clocksChebychev}
\end{align*}
where we used  Chebychev's inequality.
Besides, using a Chernoff bound, we can write, for~$a>0$,
\begin{equation}
\label{chernoffBound}
\Pm\Big(\,\abs{\eta_0}>n^d\,\Big)
\ \leqslant\ \Em\Big(e^{a\abs{\eta_0}}\Big)e^{-an^d}
\ =\ \exp\Big((e^a-1)\mu n^d-a n^d\Big)\,.
\end{equation}
Since~$\mu<1$ and~$(e^a-1)\mu-a\sim(\mu-1)a$ when~$a\to 0$, we may choose~$a>0$ such that~$(e^a-1)\mu-a<0$.
Also, we take~$c_1>0$ given by Lemma~\ref{lemmaSlowPhaseTopplings}, so that, for~$n$ large enough,
\begin{equation}
\label{applicationLemmaSlowPhaseTopplings}
\Pml\Big(\,\norme{m_{\eta_0}^\tau}\,<\,e^{c_1n^d}\,\Big)
\ <\ e^{-c_1 n^d}\,.
\end{equation}
Plugging~(\ref{clocksChebychev}),~(\ref{chernoffBound}) and~(\ref{applicationLemmaSlowPhaseTopplings}) into~(\ref{conditionningPoisson}), and choosing~$c>0$ such that
$$c\ <\ \min\big(a-(e^a-1)\mu,\,c_1\big)\,,$$
we get, for~$n$ large enough,
$$\PARW\Big(\,\calT_n\,<\,e^{cn^d}\,\Big)
\ \leqslant\ n^d (1+\lambda)e^{-2cn^d}+e^{-(a-(e^a-1)\mu)n^d}+e^{-c_1 n^d}
\ <\ e^{-cn^d}\,.$$
Decreasing the constant~$c$ if necessary so that the result holds for all~$n\geqslant 1$, we obtain Theorem~\ref{thmSlowPhase}.
\end{proof}

\section{From the torus to the infinite lattice}
\label{sectionProofThmLink}

We now turn to the proof of Theorem~\ref{thmLink}, which shows that the fixation time on the torus in the sub-critical regime cannot satisfy~(\ref{tempsExp}).
Before this, we start with two elementary results about the stabilization time of a box with open boundary condition and about Internal Diffusion Limited Aggregation.

\subsection{Upper bound on the stabilization time}
\label{sectionStabTime}

We start by proving the following rough estimate:

\begin{lemma}
\label{lemmaStabTime}
We have the following upper bound on the number~\smash{$\norme{m_{\Lambda_n,\,\eta_0}^\tau}$} of toppling instructions necessary to stabilize the box~$\Lambda_n\subset\Z^d$ with particles being killed when they jump out of~$\Lambda_n$:
$$\forall\,b>0\quad
\forall\,\mu\in(0,\,1)\qquad
\Pml\Big(\norme{m_{\Lambda_n,\,\eta_0}^\tau}>e^{bn^d}\Big)
\ =\ O\left(\frac{n^{1+3d/2}}{e^{bn^d/2}}\right)\,.$$
\end{lemma}

\begin{proof}
Let~$b>0$ and~$\mu\in(0,\,1)$.
Let us write~$c_b=2+b/2$.
Conditioning on the initial configuration~$\eta_0$, we can write
\begin{align*}
\Pml\Big(\norme{m_{\Lambda_n,\,\eta_0}^\tau}>e^{bn^d}\Big)
&\ =\ \sum_{\eta\in\N^{\Lambda_n}}\Pm\big(\eta_0=\eta\big)
\,\Pl\Big(\norme{m_{\Lambda_n,\,\eta}^\tau}>e^{bn^d}\Big)\\
&\ \leqslant\ \Pm\Big(\abs{\eta_0}>c_b n^d\Big)\,+\,
\sup_{\eta\in\N^{\Lambda_n}\,:\,|\eta|\leqslant c_b n^d}\,\Pl\Big(\norme{m_{\Lambda_n,\,\eta}^\tau}>e^{bn^d}\Big)\,.
\end{align*}
To handle the first term, we use a Chernoff bound, recalling that~$\abs{\eta_0}$ follows a Poisson distribution with mean~$\mu n^d$.
Recalling that~$\Em$ denotes the expectation relative to~$\Pm$, we have
$$\Pm\Big(\abs{\eta_0}>c_b n^d\Big)
\ \leqslant\ \Em\Big(e^{\abs{\eta_0}}\Big)e^{-c_b n^d}
\ =\ e^{(e\mu-\mu-c_b)n^d}
\ \leqslant\ e^{-bn^d/2}\,.$$
Thus, it remains to show that
\begin{equation}
\label{supEtaNorme}
\sup_{\eta\in\N^{\Lambda_n}\,:\,|\eta|\leqslant c_b n^d}\,\Pl\Big(\norme{m_{\Lambda_n,\,\eta}^\tau}>e^{bn^d}\Big)
\ =\ O\left(\frac{n^{1+3d/2}}{e^{bn^d/2}}\right)\,.
\end{equation}
To do so, we fix~$\eta\in\N^{\Lambda_n}$ such that~$\abs{\eta}\leqslant c_b n^d$, and we consider the number~$J_n$ of jump instructions (i.e., topplings which are not sleep instructions) which are necessary to stabilize the box~$\Lambda_n$ with particles ignored once they jump out of this box.
Writing~\smash{$M=\big\lceil e^{bn^d}\big\rceil$} and~\smash{$M'=\big\lfloor e^{bn^d}/(2+2\lambda)\big\rfloor$}, we have
\begin{equation}
\label{boundJT}
\Pl\Big(\norme{m_{\Lambda_n,\eta}^\tau}\geqslant M\Big)
\ \leqslant\ \Pl\big(J_n> M'\big)
\,+\,\Pl\Big(\,\big\{\norme{m_{\Lambda_n,\eta}^\tau}\geqslant M\big\}\cap\big\{J_n\leqslant M'\big\}\,\Big)\,.
\end{equation}
If~$\norme{m_{\Lambda_n,\eta}^\tau}\geqslant M$ but~$J_n\leqslant M'$, then among the first~$M$ topplings revealed, at most~$M'$ are jump instructions.
Thus, if~$X_1,\,\ldots,\,X_M$ are i.i.d.\ Bernoulli random variables with parameter~$p=1/(1+\lambda)$, we have
$$\Pl\Big(\,\big\{\norme{m_{\Lambda_n,\eta}^\tau}\geqslant M\big\}\cap\big\{J_n\leqslant M'\big\}\,\Big)
\ \leqslant\ \Proba\big(X_1+\cdots+X_M\leqslant M'\big)\,,$$
whence, by Chebychev's inequality,
\begin{equation}
\label{probaJT}
\Pl\Big(\,\big\{\norme{m_{\Lambda_n,\eta}^\tau}\geqslant M\big\}\cap\big\{J_n\leqslant M'\big\}\,\Big)
\ \leqslant\ \frac{\mathrm{Var}(X_1+\cdots+X_M)}{(Mp-M')^2}
\ \leqslant\ \frac{Mp(1-p)}{(Mp/2)^2}
\ =\ \frac{4\lambda}{M}
\ \leqslant\ 4\lambda e^{-bn^d}\,.
\end{equation}
We now deal with the other term~\smash{$\Pl\big(J_n> M'\big)$} appearing in~(\ref{boundJT}).
In the initial configuration~$\eta$, let us label the particles present at each site~$x\in\Lambda_n$ with labels~$i\in\{1,\,\ldots,\,\eta(x)\}$.
When we topple a site containing more than one particle, let us say that we topple the particle carrying the minimal label, with respect to an arbitrary order on~$\Lambda_n\times\N$.
This allows us to define, for each particle labelled~$(x,\,i)$, the number~$j^{x,i}$ of jumps (i.e., topplings which are not sleep instructions) it makes during the stabilization, and its trajectory~$X^{x,i}(0)=x,\,\ldots,\,X^{x,i}(j^{x,i})$, where~$X^{x,i}(m)$ is the position of the particle labelled~$(x,\,i)$ after~$m$ jump instructions.
With these notations, we have
$$J_n\ =\ \sum_{x\in\Lambda_n}\,\sum_{i=1}^{\eta(x)}\,j^{x,i}\,.$$
Given that~$|\eta|\leqslant c_b n^d$, this implies that
$$\Pl\big(J_n>M'\big)
\ \leqslant\ \sum_{x\in\Lambda_n}\,\sum_{i=1}^{\eta(x)}\,
\Pl\left(j^{x,i}>\frac{M'}{c_b n^d}\right)\,.$$
Let us complete the trajectory of each particle with independent jumps to obtain, for each particle labelled~$(x,\,i)$, a trajectory~\smash{$(X^{x,i}(m))_{m\in\N}\in(\Z^d)^\N$}.
Thus, for each particle labelled~$(x,\,i)$, the random walk~$X^{x,i}$ is distributed as a symmetric random walk on~$\Z^d$.
We then note that if we have~$j^{x,i}>M'/(c_b n^d)$, it implies that~\smash{$X^{x,i}\big(\Ent{M'/(c_b n^d)}\big)\in\Lambda_n$}, whence
\begin{equation}
\label{probaJ}
\Pl\big(J_n>M'\big)
\ \leqslant\ \sum_{x\in\Lambda_n}\,\sum_{i=1}^{\eta(x)}\,\Pl\Bigg(X^{x,i}\left(\Ent{\frac{M'}{c_b n^d}}\right)\in\Lambda_n\Bigg)
\ \leqslant\ c_b n^d\,P_0^{\Z^d}\Bigg(\norme{X\left(\Ent{\frac{M'}{c_b n^d}}\right)}_\infty\leqslant\,n\Bigg)\,,
\end{equation}
where, under~$P_0^{\Z^d}$,~$X$ is a simple random walk on~$\Z^d$ started at the origin.
By the pigeonhole principle, among the first~\smash{$\Ent{M'/(c_b n^d)}$} steps of a random walk on~$\Z^d$, at least~\smash{$\Ent{M'/(d c_b n^d)}$} of them are in the same direction.
Therefore, we have
$$P_0^{\Z^d}\Bigg(\norme{X\left(\Ent{\frac{M'}{c_b n^d}}\right)}_\infty\leqslant\, n\Bigg)
\ \leqslant\ d\, P_0^\Z\Bigg(\abs{X\left(\Ent{\frac{M'}{d c_b n^d}}\right)}\,\leqslant\, n\Bigg)\,.$$
A classical computation shows that, when~$k\to\infty$,
$$P_0^\Z\Big(\big|X(k)\big|\leqslant n\Big)
\ \leqslant\ \frac{n+1}{2^{k}}\binom{k}{\Ent{k/2}}
\ \stackrel{k\to\infty}{\sim}\ \sqrt{\frac{2}{\pi}}\frac{n+1}{\sqrt{k}}\,,$$
whence, recalling that~\smash{$M'=\big\lfloor e^{bn^d}/(2+2\lambda)\big\rfloor$},
$$P_0^{\Z^d}\Bigg(\norme{X\left(\Ent{\frac{M'}{c_b n^d}}\right)}_\infty\leqslant \,n\Bigg)
\ =\ O\big(n^{1+d/2}e^{-bn^d/2}\big)\,.$$
Plugging this into~(\ref{probaJ}), we obtain
\begin{equation}
\label{probaJ2}
\Pl\big(J_n>M'\big)
\ =\ O\big(n^{1+3d/2}e^{-bn^d/2}\big)\,,
\end{equation}
the bound being uniform with respect to the initial configuration~$\eta$.
Replacing~(\ref{probaJT}) and~(\ref{probaJ2}) into~(\ref{boundJT}), we obtain the desired bound~(\ref{supEtaNorme}), concluding the proof.
\end{proof}

\subsection{A bound on Internal Diffusion Limited Aggregation}
\label{sectionIDLA}

In the proof of Theorem~\ref{thmLink}, we will use a rough estimate about a related model called Internal Diffusion Limited Aggregation, which corresponds to the Activated Random Walk model with an infinite sleep rate.
In this model, active particles fall asleep instantaneously as soon as they are alone on a site.

Starting from an initial configuration of active particles~\smash{$\eta\in\N^{\Z^d}$} with a finite number of particles, the model almost surely reaches a stable configuration which we write~$\sigma_{\eta}^\tau$, as defined in section~\ref{sectionAbelian}.
The result we need is the following Lemma, which is a weak version of Lemma~A of~\cite{JLS12} (note that~$\Pinf$ is simply~$\Pl$ with~$\lambda=\infty$):

\begin{lemma}
\label{lemmaIDLA}
For~$\beta,\,R\in(0,\,\infty)$, let us denote by~$B(0,\,R)$ the Euclidian ball of radius~$R$ centered on the origin, and let~$\calC_{\beta,R}$ be the set of all possible initial confi\smash{gu}rations with not more than~$\beta R^d$ particles and no particle inside~$B(0,\,R)$, that is to say:
$$\calC_{\beta,R}
\ =\ \Big\{\,\eta\in\N^{\Z^d}\ :\ \abs{\eta}\leqslant \beta R^d\quadet\forall\,x\in \Z^d\cap B(0,\,R)\quad \eta(x)=0\,\Big\}\,.$$
Then there exists~$K=K(d)>0$ such that, for~$\beta$ small enough, when~$d\geqslant 3$, we have
$$\sup_{\eta\in\calC_{\beta,R}}\,
\Pinf\big(\sigma_\eta^\tau(0)>0\big)
\ \leqslant\ e^{-K R^2}\,,$$
whereas for~$d=2$ we have
$$\sup_{\eta\in\calC_{\beta,R}}\,
\Pinf\big(\sigma_\eta^\tau(0)>0\big)
\ \leqslant\ \exp\left(-K \frac{R^2}{\ln R}\right)\,.$$
\end{lemma}

Note that in dimension~$d=1$, we simply have~\smash{$\Pinf\big(\sigma_\eta^\tau(0)>0\big)=0$} for every initial configuration~$\eta\in\calC_{1,R}$.

In fact, we do not need a bound as sharp as above, and we only use that, whatever the dimension~$d\geqslant 1$, we have, for~$\beta$ small enough,
\begin{equation}
\label{lemmaIDLAweak}
\forall\,R\geqslant 1\qquad
\sup_{\eta\in\calC_{\beta,R}}\,
\Pinf\big(\sigma_\eta^\tau(0)>0\big)
\ \leqslant\ e^{-K(d) R}\,.
\end{equation}

\subsection{Proof of Theorem~\ref{thmLink}}
\label{sectionProofThmLinkFinal}

\begin{proof}[Proof of Theorem~\ref{thmLink}]
Let~$\lambda\in(0,\,\infty)$ and let~$\mu<\mu_c(\lambda)$.
We want to prove that the property~(\ref{tempsExp}) does not hold on the torus.
Hence, we have to prove that
\begin{equation}
\label{nonTempsExp}
\forall\,c>0\quad
\exists\,n\geqslant 1\qquad
\PARW\Big(\,\calT_n\,<\,e^{cn^d}\,\Big)
\ \geqslant\ e^{-cn^d}\,,
\end{equation}
where~$\calT_n$ is the stabilization time of the model on the torus.
Thus, we start by fixing an arbitrary real number~$c>0$.
Let~$\beta>0$ be given by Lemma~\ref{lemmaIDLA} about Internal Diffusion Limited Aggregation, and let us choose two real parameters~$a$ and~$\varepsilon$ such that
\begin{equation}
\label{conditionAlpha}
0\ <\ a\ <\ \frac{c}{96d\mu}\,\wedge\,\frac{c}{72d\ln(1+1/\lambda)}\,\wedge\,\frac{1}{3}
\end{equation}
and
\begin{equation}
\label{conditionEpsilon}
0\ <\ \varepsilon\ <\ \frac{\beta a^d}{4^d}\,\wedge\,\frac{c}{24}\,.
\end{equation}
As explained in section~\ref{sectionSketchThmLink}, we consider four sizes of sub-boxes of the torus, by writing~\smash{$p:\Z^d\to\torus$} for the canonical projection application and by considering the projection of the boxes~$\Lambda_k$ defined by~(\ref{defBoite}).
We call~\smash{$\boxO=\torus=p(\Lambda_n)$} the whole box,~\smash{$\boxI=p(\Lambda_{n-\Ent{an}})$} the medium box,~\smash{$\boxII=p(\Lambda_{n-2\Ent{an}})$} the small box and~\smash{$\boxIII=p(\Lambda_{n-3\Ent{an}})$} the tiny box.
For every box~$\Lambda_k$, we denote its internal boundary by
$$\partial^i\Lambda_k
\ =\ \Big\{\,x\in\Lambda_k\ :\ \exists\, y\in\Z^d\setminus\Lambda_k\quad x\sim y\,\Big\}
\ =\ \Lambda_k\setminus\Lambda_{k-2}\,,$$
and its external boundary by
$$\partial^e\Lambda_k\ =\ \Big\{\,x\in\Z^d\setminus\Lambda_k\ :\ \exists\,y\in\Lambda_k\quad x\sim y\,\Big\}\,,$$
and, for every~$j\in\{0,\,1,\,2,\,3\}$, we write~\smash{$\partial^i B_j=p(\partial^i\Lambda_{n-j\Ent{an}})$} and~\smash{$\partial^e B_j=p(\partial^e\Lambda_{n-j\Ent{an}})$} for the projections on the torus of these boundaries.

As explained in paragraph~\ref{sectionSketchThmLink}, the idea is to assume that we start with all the particles inside the small box~\smash{$\boxII$}, then to stabilize this small box, and then to try to stabilize the particles which escaped from this box, without waking up the particles which are sleeping inside of the tiny box~\smash{$\boxIII$}.

We now describe a toppling procedure, which, when successful, yields a sequence of topplings which stabilizes the configuration in the torus~\smash{$\boxO$}, with no particle ever visiting the boundary~\smash{$\partial^i\boxO$}.
During this procedure, we allow ourselves to use acceptable topplings, as defined in section~\ref{sectionSitewise}.
Thanks to the monotonicity property of Lemma~\ref{lemmaMonotonicity}, this yields an upper bound on the number of topplings necessary to stabilize the torus with only legal topplings and, thereby, on the stabilization time of the torus in the continuous-time process.

Let us now describe our toppling procedure, which is composed of several steps, represented on figure~\ref{figTopplingThmLink}.
At each step, in certain cases, the procedure can fail, in which case we stop the procedure.

\begin{itemize}
\item Step~0 is to draw the initial configuration~$\eta_0$ according to~$\Pm$, that is to say, with i.i.d.\ Poisson numbers of active particles with mean~$\mu$ on each site.
We say that step~0 is successful if there are no particles outside of the small box~\smash{$\boxII$}, i.e., if~$\eta_0(x)=0$ for all~\smash{$x\in\boxO\setminus\boxII$}.

\item In the first step, we stabilize the small box~\smash{$\boxII$}, with particles being stopped in place upon jumping out of~\smash{$\boxII$}.
We say that this step is successful if it terminates with a finite number of topplings and at most~\smash{$\varepsilon n^d$} particles jump out of the small box.
Then, we obtain a configuration with some sleeping particles in~\smash{$\boxII$}, and at most~\smash{$\varepsilon n^d$} active particles all located on the external boundary~\smash{$\partial^e\boxII$}.

\item Then, during the second step, the idea is to force all the particles left on~\smash{$\partial^e\boxII$} to walk until they jump out of the medium box~\smash{$\boxI$}.
It is during this step that we use acceptable topplings.
First, we choose a non-empty site of~\smash{$\partial^e\boxII$} and we topple it.
If it was a sleep instruction, we topple it again (this might not be a legal toppling but it is acceptable), until we get a jump instruction to one of the neighbouring sites.
We then topple this neighbouring site, until we get a jump instruction, and so on, until we reach a site of~\smash{$\partial^i\boxI$}.
We then repeat this procedure until all the sites of~\smash{$\partial^e\boxII$} are empty.

Although formally we are not allowed to say that we topple a specific particle instead of a site, one can think of this step as forcing a particle to walk, ignoring the sleep instructions and ignoring all the other particles.
Thus, during this step, some particles which were asleep in the small box~\smash{$\boxII$} after step~1 might wake up, but the number of particles on each site of this small box is left unchanged after step~2.

We say that this step~2 is successful if it terminates in a finite number of topplings and if no active particle ever visits the tiny box~\smash{$\boxIII$} during this step.
In this case, we obtain a configuration with some sleeping particles in the tiny box~\smash{$\boxIII$}, some sleeping and some active particles in~\smash{$\boxII\setminus\boxIII$} (because some particles left asleep there after step~1 may have been awaken during step~2), and at most~\smash{$\varepsilon n^d$} active particles on the boundary~\smash{$\partial^e\boxI$}, with no particles elsewhere.
Note in particular that, after step~2, no site of the small box~\smash{$\boxII$} can contain more than one particle.

\item During step~3, as long as there is a site containing at least two particles, we topple it.
We say that this step is successful if it terminates and if no particle ever visits neither~\smash{$\partial^i\boxO$} nor~\smash{$\partial^e\boxII$} during this step.
In this case, the configuration inside the small box~\smash{$\boxII$} does not change during this step, and we end up with at most~\smash{$\varepsilon n^d$} sites of~\smash{$\boxO\setminus\boxII$} with exactly one active particle, the other sites of~\smash{$\boxO\setminus\boxII$} being empty.

\item During the fourth and last step, we topple exactly once each of the unstable sites.
We say that step~4 is successful if all the topplings revealed during this step are sleep instructions.
Since we start step~4 with at most one particle on each site, if this step is successful it leads to a stable configuration, with all the particles asleep.
\end{itemize}

We now estimate the probability that each of these steps is successful.
For each~$i\in\{0,\,\ldots,\,4\}$, we denote by~$\calS_i$ the event that the steps~$0$ to~$i$ are all successful.

First, we estimate~$\Pm(\calS_0)$.
At each site~\smash{$x\in\boxO$}, the initial number of particles~$\eta_0(x)$ is drawn according to a Poisson distribution with parameter~$\mu$, the number of particles at distinct sites being independent, whence
\begin{equation}
\label{probaE0}
\Pm\big(\calS_0\big)
\ =\ \big(e^{-\mu}\big)^{\abs{\boxO\setminus\boxII}}
\ \geqslant\ e^{-4d\mu \Ent{an} n^{d-1}}
\ \geqslant\ e^{-4d\mu a n^d}
\ \geqslant\ e^{-cn^d/24}\,,
\end{equation}
where we have used our assumption~(\ref{conditionAlpha}) on the parameter~$a$ in the last inequality.

We now turn to the event~$\calS_1$.
Writing~\smash{$M_{n-2\Ent{an}}$} for the number of particles which jump out of the small box~\smash{$\boxII$} during step~1, it follows from Lemma~\ref{lemmaMn} that, when~$n\to\infty$,
$$\Eml(M_{n-2\Ent{an}})
\ =\ o\big(|B_2|\big)\,.$$
In fact, Lemma~\ref{lemmaMn} deals with the model on~$\Z^d$ instead of the torus but, if one is only interested in the number of particles which jump out of a square box of side~\smash{$n-2\Ent{an}$} when stabilizing only inside this box with particles ignored once they jump out of it, the fact that this box is included in a torus or in~$\Z^d$ has no importance.
Thus, for~$n$ large enough, we have
$$\Eml(M_{n-2\Ent{an}})
\ \leqslant\ \frac{\varepsilon}{4(1-2a)^d}\abs{\boxII}
\ \eqninfty\ \frac{\varepsilon n^d}{4}\,,$$
which implies by Markov's inequality that, for~$n$ large enough,
$$\Pml\Big(M_{n-2\Ent{an}}\,\leqslant\,\varepsilon n^d\Big)
\ \geqslant\ \demi\,.$$
This event being independent of~$\calS_0$, combining this with~(\ref{probaE0}) we get
\begin{equation}
\label{probaE1}
\Pml\big(\calS_1\big)
\ =\ \Pml\Big(\,\calS_0\cap\big\{M_{n-2\Ent{an}}\,\leqslant\,\varepsilon n^d\big\}\,\Big)
\ \geqslant\ \demi\,e^{-cn^d/24}\,.
\end{equation}
Next, we come to the event~$\calS_2$ that the procedure is successful until step~2 included.
Let us denote by~\smash{$\eta_1:\boxO\to\N_\s$} the configuration reached after step~1, with sleeping particles inside~\smash{$\boxII$} and~\smash{$M_{n-2\Ent{an}}$} active particles on the boundary~\smash{$\partial^e\boxII$}.
For every~\smash{$x\in\partial^e\boxII$}, the probability that a particle starting at~$x$ reaches the boundary~\smash{$\partial^e\boxI$} before reaching the tiny box~\smash{$\boxIII$} is the probability that a simple random walk on~$\Z^d$ hits~\smash{$\partial^e\boxI$} before~\smash{$\boxIII$} when started at~\smash{$x\in\partial^e\boxII$}, which by symmetry is at least~$1/2$.
Hence, for every~\smash{$\eta:\partial^e\boxII\to\N$} such that~$\abs{\eta}\leqslant\varepsilon
 n^d$ (so that we condition on an event which has positive probability), we have
$$\Pml\Big(\calS_2\ \Big|\ \calS_1\cap\big\{\eta_1\restrict{\partial^e\boxII}=\eta\big\}\,\Big)
\ \geqslant\ \frac{1}{2^{\abs{\eta}}}
\ \geqslant\ \frac{1}{2^{\varepsilon n^d}}
\ \geqslant\ e^{-cn^d/24}\,,$$
where in the last inequality we used our assumption~(\ref{conditionEpsilon}) on~$\varepsilon$ which ensures that~$\varepsilon\ln 2<\varepsilon<c/24$.
This bound being uniform with respect to~$\eta$, we deduce that
$$\Pml\big(\calS_2\,\big|\,\calS_1\big)\ \geqslant\ e^{-cn^d/24}\,.$$
Combining this with our previous estimate~(\ref{probaE1}), we obtain 
\begin{equation}
\label{probaE2}
\Pml\big(\calS_2\big)
\ \geqslant\ \demi\,e^{-cn^d/12}\,.
\end{equation}

We now study the probability of success of the third step.
Recall that this step starts with a certain number of active particles on~\smash{$\partial^e\boxI$}, the other particles all being alone on their sites inside the small box~\smash{$\boxII$} (see figure~\ref{figTopplingThmLink}).

Recall also that, during this step, we only topple the sites with at least two particles.
Thus, the dynamic on the particles, regardless of their state, active or sleeping, exactly corresponds to Internal Diffusion Limited Aggregation, as described in paragraph~\ref{sectionIDLA}.
Furthermore, we can note that, as soon as a particle reaches a site of~\smash{$\partial^i\boxO\cup\partial^e\boxII$}, step~3 can be stopped and declared unsuccessful.
Doing so, no site of the small box~\smash{$\boxII$} is ever toppled during step~3, so that, during this step, we may forget the particles which lie in~\smash{$\boxII$}.

Conditioned on the realization of~$\calS_2$ and on the configuration~$\eta_2$ reached after step~2, the probability of~$\calS_3$ is given by the probability that the Internal Diffusion Limited Aggregation (IDLA) starting with~$\eta_2$ never reaches neither the internal boundary of the whole box,~\smash{$\partial^i\boxO$} nor the external boundary of the small box,~\smash{$\partial^e\boxII$}.
With the above remark, we can see that this probability does not change if, instead of~$\eta_2$, we start the IDLA with the configuration~$\eta_2$ restricted to~\smash{$\partial^e\boxI$} (because the other particles of~$\eta_2$ lie inside the small box~\smash{$\boxII$} and thus have no influence on the probability of success of step~3).
Therefore, fixing a configuration~\smash{$\eta:\partial^e\boxI\to\N$} such that~$\abs{\eta}\leqslant\varepsilon n^d$, we may write
\begin{align*}
\Pml\Big(\,\calS_3\ \Big|\ \calS_2\cap\big\{\eta_2\restrict{\partial^e\boxI}=\eta\big\}\,\Big)
&\ =\ \calP^\infty\Big(\,\forall\, y\in\partial^i\boxO\cup\partial^e\boxII\quad \sigma_\eta^{\tau}(y)\,=\,0\,\Big)\\
&\ \geqslant\ 1-\sum_{y\in\partial^i\boxO\cup\partial^e\boxII}
\calP^\infty\Big(\,\sigma_\eta^{\tau}(y)\,>\,0\,\Big)\,.\numberthis\label{probaE3cond}
\end{align*}
Yet, defining~$R_n=an/2-2$, it follows from our assumption~(\ref{conditionEpsilon}) on~$\varepsilon$ that, for~$n$ large enough,
$$\varepsilon n^d
\ \leqslant\ \frac{\beta a^d n^d}{4^d}
\ \leqslant\ \beta\left(\frac{an}{2}-2\right)^d
\ =\ \beta R_n^d\,,$$
implying that~$\abs{\eta}\leqslant \beta R_n^d$.
Thus, we may apply Lemma~\ref{lemmaIDLA}, or rather its weakened version~(\ref{lemmaIDLAweak}), to obtain that for every point~\smash{$y\in\partial^i\boxO\cup\partial^e\boxII$}, we have
$$\calP^\infty\Big(\,\sigma_\eta^{\tau}(y)\,>\,0\,\Big)
\ \leqslant\ e^{-K R_n}\,.$$
Plugging this into~(\ref{probaE3cond}), we obtain that
$$\Pml\Big(\,\calS_3\ \Big|\ \calS_2\cap\big\{\eta_2\restrict{\partial^e\boxI}=\eta\big\}\,\Big)
\ \geqslant\ 1-\abs{\partial^i\boxO\cup\partial^e\boxII}e^{-KR_n}
\ \geqslant\ 1-4d n^{d-1} e^{2K-Kan}
\ \geqslant\ \demi\,,$$
for~$n$ large enough.
This bound being uniform with respect to~$\eta$, we deduce that, for~$n$ large enough,
$$\Pml\big(\calS_3\,\big|\,\calS_2\big)\ \geqslant\ \demi\,.$$
Combined with our previous estimate~(\ref{probaE2}) on the probability of~$\calS_2$, this yields
\begin{equation}
\label{probaE3}
\Pml\big(\calS_3\big)
\ \geqslant\ \frac{1}{4}\,e^{-cn^d/12}\,.
\end{equation}
We now deal with the event~$\calS_4$.
At the beginning of step~4, the unstable sites are the sites where the particles spread during step~3, which are all in~\smash{$\boxO\setminus\boxII$}, and the sites of the particles in~\smash{$\boxII\setminus\boxIII$} which have been awaken during step~2.
In any case, there is no unstable site in the tiny box~\smash{$\boxIII$}.
Thus, the number of topplings to perform during step~4 does not exceed
$$\abs{\boxO\setminus\boxIII}
\ \leqslant\ 3d\Ent{an} n^{d-1}
\ \leqslant\ 3dan^d\,,$$
whence, recalling our assumption~(\ref{conditionAlpha}) on the parameter~$a$,
$$\Pml\big(\calS_4\,\big|\,\calS_3\big)
\ \geqslant\ \left(\frac{\lambda}{1+\lambda}\right)^{3dan^d}
\ \geqslant\ e^{-cn^d/24}\,.$$
Combining this with~(\ref{probaE3}), we get
\begin{equation}
\label{probaE4}
\Pml\big(\calS_4\big)
\ \geqslant\ \frac{1}{4}\,e^{-cn^d/8}\,.
\end{equation}
As explained above, if the event~$\calS_4$ is realized, then the whole procedure is successful and it yields an acceptable sequence of topplings~$\alpha$ which stabilizes the initial configuration~$\eta_0$, and which is such that the sites on the boundary~\smash{$\partial^i\boxO$} are never toppled.
Thanks to the monotonicity property given by Lemma~\ref{lemmaAcceptable}, this implies that~$m_{\eta_0}^\tau(x)=0$ for all~\smash{$x\in\partial^i\boxO$}, where we recall that~$m_{\eta_0}^\tau$ is the odometer of a legal sequence which stabilizes~$\eta$ in the torus using the toppling instructions in~$\tau$, as defined in section~\ref{sectionSitewise}.
In this case,~$m_{\eta_0}^\tau$ is a legal sequence of topplings which only topples the sites of~$\boxU=\boxO\setminus\partial^i\boxO$ and which stabilizes the configuration~$\eta_0$ on the torus, and, especially, it also stabilizes~$\eta_0$ on~$\boxU$.
Hence, by the Abelian property (Lemma~\ref{lemmaAbelian}), the occurrence of~$\calS_4$ implies that~\smash{$m_{\eta_0}^\tau=m_{\boxU,\,\eta_0}^\tau$}.
Thus, our result~(\ref{probaE4}) becomes
\begin{equation}
\label{probaOdometerEquals}
\Pml\big(m_{\eta_0}^\tau=m_{\boxU,\,\eta_0}^\tau\big)
\ \geqslant\ \frac{1}{4}\,e^{-cn^d/8}\,.
\end{equation}
Besides, Lemma~\ref{lemmaStabTime} tells us that
$$\Pml\Big(\big\|m_{\boxU,\,\eta_0}^\tau\big\|>e^{cn^d/2}\Big)
\ =\ O\left(\frac{n^{1+3d/2}}{e^{c(n-2)^d/4}}\right)\,,$$
whence, for~$n$ large enough,
$$\Pml\Big(\big\|m_{\boxU,\,\eta_0}^\tau\big\|>e^{cn^d/2}\Big)
\ \leqslant\ \frac{1}{8}\,e^{-cn^d/8}\,.$$
Combining this with~(\ref{probaOdometerEquals}), we have
\begin{align*}
\Pml\Big(\big\|m_{\eta_0}^\tau\big\|\leqslant e^{cn^d/2}\Big)
&\ \geqslant\ \Pml\Big(\big\{m_{\eta_0}^\tau=m_{\boxU,\,\eta_0}^\tau\big\}\cap\big\{\big\|m_{\boxU,\,\eta_0}^\tau\big\|\leqslant e^{cn^d/2}\big\}\Big)\\
&\ \geqslant\ \Pml\big(m_{\eta_0}^\tau=m_{\boxU,\,\eta_0}^\tau\big)
-\Pml\Big(\big\|m_{\boxU,\,\eta_0}^\tau\big\|>e^{cn^d/2}\Big)
\ \geqslant\ \frac{1}{8}\,e^{-cn^d/8}\,.\numberthis\label{probaNtopplings}
\end{align*}
We now turn to the stabilization time~$\calT_n$ of the continuous-time process on the torus, and we write
\begin{align*}
\PARW\Big(\,\calT_n\,<\,e^{cn^d}\,\Big)
&\ \geqslant\ \PARW\Big(\,\big\{\calT_n\,<\,e^{cn^d}\big\}\cap\big\{\big\|m_{\eta_0}^\tau\big\|\leqslant e^{cn^d/2}\big\}\,\Big)\\
&\ =\ \Pml\Big(\big\|m_{\eta_0}^\tau\big\|\leqslant e^{cn^d/2}\Big)
\,-\,\PARW\Big(\,\big\{\calT_n\,\geqslant\,e^{cn^d}\big\}\cap\big\{\big\|m_{\eta_0}^\tau\big\|\leqslant e^{cn^d/2}\big\}\,\Big)\,.\numberthis\label{clocksFinal}
\end{align*}
If we have~\smash{$\calT_n\geqslant e^{cn^d}$} while~\smash{$\big\|m_{\eta_0}^\tau\big\|\leqslant e^{cn^d/2}$}, it means that, during the time interval~\smash{$[0,\,e^{cn^d}]$} the Poisson clocks triggering a toppling event rang at most~\smash{$e^{cn^d/2}$} times, although there were always at least one active particle during this time frame, which entails that a toppling event happens with rate at least~$1+\lambda$.
Hence, taking~$\theta\subset\R_+$ a Poisson process of intensity~$1+\lambda$, we have, using Chebychev's inequality,
\begin{align*}
\PARW\Big(\,\big\{\calT_n\,\geqslant\,e^{cn^d}\big\}\cap\big\{\big\|m_{\eta_0}^\tau\big\|\leqslant e^{cn^d/2}\big\}\,\Big)
&\ \leqslant\ \Proba\Big(\abs{\theta\cap[0,\,e^{cn^d}]}\leqslant e^{cn^d/2}\Big)\\
&\ \leqslant\ \frac{\mathrm{Var}\Big(\abs{\theta\cap[0,\,e^{cn^d}]}\Big)}{\big(e^{cn^d}-e^{cn^d/2}\big)^2}
\ =\ \frac{e^{cn^d}}{\big(e^{cn^d}-e^{cn^d/2}\big)^2}
\ =\ O\big(e^{-cn^d}\big)\,.
\end{align*}
whence, for~$n$ large enough,
\begin{equation}
\label{chebychev}
\PARW\Big(\,\big\{\calT_n\,\geqslant\,e^{cn^d}\big\}\cap\big\{\big\|m_{\eta_0}^\tau\big\|\leqslant e^{cn^d/2}\big\}\,\Big)
\ \leqslant\ \frac{1}{16}\,e^{-cn^d/8}\,.
\end{equation}
Plugging~(\ref{probaNtopplings}) and~(\ref{chebychev}) into~(\ref{clocksFinal}), we obtain that, for~$n$ large enough,
$$\PARW\Big(\,\calT_n\,<\,e^{cn^d}\,\Big)
\ \geqslant\ \frac{1}{16}\,e^{-cn^d/8}\,,$$
which proves~(\ref{nonTempsExp}), concluding the proof of Theorem~\ref{thmLink}.
\end{proof}

%
%

\begin{acks}[Acknowledgments]
A.\ G.\ thanks Leonardo Rolla for his introduction to the model
and his kind hospitality at NYU Shanghai where this work could start.
A.\ G.\ and N.\ F.\ thank the referee for his comments which helped to clarify several parts of the manuscript.
\end{acks}
%


\bibliographystyle{imsart-number} 
\bibliography{arw.bib}       

\begin{thebibliography}{34}

\bibitem{AMP02a}
\begin{barticle}[author]
\bauthor{\bsnm{Alves},~\bfnm{Oswaldo}\binits{O.}},
  \bauthor{\bsnm{Machado},~\bfnm{Fabio}\binits{F.}} \AND
  \bauthor{\bsnm{Popov},~\bfnm{Serguei}\binits{S.}}
(\byear{2002}).
\btitle{Phase transition for the frog model}.
\bjournal{Electronic Journal of Probability}
\bvolume{7}
\bpages{1--21}.
\end{barticle}
\endbibitem

\bibitem{AMP02b}
\begin{barticle}[author]
\bauthor{\bsnm{Alves},~\bfnm{Oswaldo~SM}\binits{O.~S.}},
  \bauthor{\bsnm{Machado},~\bfnm{Fabio~P}\binits{F.~P.}} \AND
  \bauthor{\bsnm{Popov},~\bfnm{S~Yu}\binits{S.~Y.}}
(\byear{2002}).
\btitle{The shape theorem for the frog model}.
\bjournal{The Annals of Applied Probability}
\bvolume{12}
\bpages{533--546}.
\end{barticle}
\endbibitem

\bibitem{AGG10}
\begin{barticle}[author]
\bauthor{\bsnm{Amir},~\bfnm{Gideon}\binits{G.}} \AND
  \bauthor{\bsnm{Gurel-Gurevich},~\bfnm{Ori}\binits{O.}}
(\byear{2010}).
\btitle{On fixation of activated random walks}.
\bjournal{Electron. Commun. Probab.}
\bvolume{15}
\bpages{119--123}.
\bdoi{10.1214/ECP.v15-1536}
\bmrnumber{2643591}
\end{barticle}
\endbibitem

\bibitem{ARS19}
\begin{barticle}[author]
\bauthor{\bsnm{Asselah},~\bfnm{Amine}\binits{A.}},
  \bauthor{\bsnm{Rolla},~\bfnm{Leonardo~T}\binits{L.~T.}} \AND
  \bauthor{\bsnm{Schapira},~\bfnm{Bruno}\binits{B.}}
(\byear{2019}).
\btitle{Diffusive bounds for the critical density of activated random walks}.
\bjournal{arXiv preprint arXiv:1907.12694}.
\end{barticle}
\endbibitem

\bibitem{BTW87}
\begin{barticle}[author]
\bauthor{\bsnm{Bak},~\bfnm{Per}\binits{P.}},
  \bauthor{\bsnm{Tang},~\bfnm{Chao}\binits{C.}} \AND
  \bauthor{\bsnm{Wiesenfeld},~\bfnm{Kurt}\binits{K.}}
(\byear{1987}).
\btitle{Self-organized criticality: An explanation of the 1/f noise}.
\bjournal{Physical review letters}
\bvolume{59}
\bpages{381}.
\end{barticle}
\endbibitem

\bibitem{BGH18}
\begin{barticle}[author]
\bauthor{\bsnm{Basu},~\bfnm{Riddhipratim}\binits{R.}},
  \bauthor{\bsnm{Ganguly},~\bfnm{Shirshendu}\binits{S.}} \AND
  \bauthor{\bsnm{Hoffman},~\bfnm{Christopher}\binits{C.}}
(\byear{2018}).
\btitle{Non-fixation for conservative stochastic dynamics on the line}.
\bjournal{Comm. Math. Phys.}
\bvolume{358}
\bpages{1151--1185}.
\bdoi{10.1007/s00220-017-3059-7}
\bmrnumber{3778354}
\end{barticle}
\endbibitem

\bibitem{BGHR19}
\begin{barticle}[author]
\bauthor{\bsnm{Basu},~\bfnm{Riddhipratim}\binits{R.}},
  \bauthor{\bsnm{Ganguly},~\bfnm{Shirshendu}\binits{S.}},
  \bauthor{\bsnm{Hoffman},~\bfnm{Christopher}\binits{C.}} \AND
  \bauthor{\bsnm{Richey},~\bfnm{Jacob}\binits{J.}}
(\byear{2019}).
\btitle{Activated random walk on a cycle}.
\bjournal{Ann. Inst. Henri Poincar\'{e} Probab. Stat.}
\bvolume{55}
\bpages{1258--1277}.
\bdoi{10.1214/18-aihp918}
\bmrnumber{4010935}
\end{barticle}
\endbibitem

\bibitem{CRS14}
\begin{barticle}[author]
\bauthor{\bsnm{Cabezas},~\bfnm{M.}\binits{M.}},
  \bauthor{\bsnm{Rolla},~\bfnm{L.~T.}\binits{L.~T.}} \AND
  \bauthor{\bsnm{Sidoravicius},~\bfnm{V.}\binits{V.}}
(\byear{2014}).
\btitle{Non-equilibrium phase transitions: activated random walks at
  criticality}.
\bjournal{J. Stat. Phys.}
\bvolume{155}
\bpages{1112--1125}.
\bdoi{10.1007/s10955-013-0909-3}
\bmrnumber{3207731}
\end{barticle}
\endbibitem

\bibitem{Dhar06}
\begin{barticle}[author]
\bauthor{\bsnm{Dhar},~\bfnm{Deepak}\binits{D.}}
(\byear{2006}).
\btitle{Theoretical studies of self-organized criticality}.
\bjournal{Phys. A}
\bvolume{369}
\bpages{29--70}.
\bdoi{10.1016/j.physa.2006.04.004}
\bmrnumber{2246566}
\end{barticle}
\endbibitem

\bibitem{DF91}
\begin{barticle}[author]
\bauthor{\bsnm{Diaconis},~\bfnm{Persi}\binits{P.}} \AND
  \bauthor{\bsnm{Fulton},~\bfnm{William}\binits{W.}}
(\byear{1991}).
\btitle{A growth model, a game, an algebra, Lagrange inversion, and
  characteristic classes}.
\bjournal{Rend. Sem. Mat. Univ. Pol. Torino}
\bvolume{49}
\bpages{95--119}.
\end{barticle}
\endbibitem

\bibitem{DRS10}
\begin{barticle}[author]
\bauthor{\bsnm{Dickman},~\bfnm{Ronald}\binits{R.}},
  \bauthor{\bsnm{Rolla},~\bfnm{Leonardo~T}\binits{L.~T.}} \AND
  \bauthor{\bsnm{Sidoravicius},~\bfnm{Vladas}\binits{V.}}
(\byear{2010}).
\btitle{Activated random walkers: Facts, conjectures and challenges}.
\bjournal{Journal of Statistical Physics}
\bvolume{138}
\bpages{126--142}.
\end{barticle}
\endbibitem

\bibitem{DS84}
\begin{bbook}[author]
\bauthor{\bsnm{Doyle},~\bfnm{Peter~G.}\binits{P.~G.}} \AND
  \bauthor{\bsnm{Snell},~\bfnm{J.~Laurie}\binits{J.~L.}}
(\byear{1984}).
\btitle{Random walks and electric networks}.
\bseries{Carus Mathematical Monographs}
\bvolume{22}.
\bpublisher{Mathematical Association of America, Washington, DC}.
\bmrnumber{920811}
\end{bbook}
\endbibitem

\bibitem{These}
\begin{bphdthesis}[author]
\bauthor{\bsnm{Forien},~\bfnm{Nicolas}\binits{N.}}
(\byear{2020}).
\btitle{About self-organized criticality},
\btype{PhD thesis},
\bpublisher{Universit{\'e} Paris-Saclay}.
\end{bphdthesis}
\endbibitem

\bibitem{GS09}
\begin{barticle}[author]
\bauthor{\bsnm{Gantert},~\bfnm{Nina}\binits{N.}} \AND
  \bauthor{\bsnm{Schmidt},~\bfnm{Philipp}\binits{P.}}
(\byear{2009}).
\btitle{Recurrence for the frog model with drift on Z}.
\bjournal{Markov Process. Related Fields}
\bvolume{15}
\bpages{51--58}.
\end{barticle}
\endbibitem

\bibitem{HALWBDW20}
\begin{barticle}[author]
\bauthor{\bsnm{Helmrich},~\bfnm{S}\binits{S.}},
  \bauthor{\bsnm{Arias},~\bfnm{A}\binits{A.}},
  \bauthor{\bsnm{Lochead},~\bfnm{G}\binits{G.}},
  \bauthor{\bsnm{Wintermantel},~\bfnm{TM}\binits{T.}},
  \bauthor{\bsnm{Buchhold},~\bfnm{M}\binits{M.}},
  \bauthor{\bsnm{Diehl},~\bfnm{S}\binits{S.}} \AND
  \bauthor{\bsnm{Whitlock},~\bfnm{S}\binits{S.}}
(\byear{2020}).
\btitle{Signatures of self-organized criticality in an ultracold atomic gas}.
\bjournal{Nature}
\bvolume{577}
\bpages{481--486}.
\end{barticle}
\endbibitem

\bibitem{HJJ17}
\begin{barticle}[author]
\bauthor{\bsnm{Hoffman},~\bfnm{Christopher}\binits{C.}},
  \bauthor{\bsnm{Johnson},~\bfnm{Tobias}\binits{T.}} \AND
  \bauthor{\bsnm{Junge},~\bfnm{Matthew}\binits{M.}}
(\byear{2017}).
\btitle{Recurrence and transience for the frog model on trees}.
\bjournal{The Annals of Probability}
\bvolume{45}
\bpages{2826--2854}.
\end{barticle}
\endbibitem

\bibitem{HRR20}
\begin{barticle}[author]
\bauthor{\bsnm{Hoffman},~\bfnm{Christopher}\binits{C.}},
  \bauthor{\bsnm{Richey},~\bfnm{Jacob}\binits{J.}} \AND
  \bauthor{\bsnm{Rolla},~\bfnm{Leonardo~T}\binits{L.~T.}}
(\byear{2020}).
\btitle{Active phase for activated random walk on~$\mathbb{Z}$}.
\bjournal{arXiv preprint arXiv:2009.09491}.
\end{barticle}
\endbibitem

\bibitem{Jensen98}
\begin{bbook}[author]
\bauthor{\bsnm{Jensen},~\bfnm{Henrik~Jeldtoft}\binits{H.~J.}}
(\byear{1998}).
\btitle{Self-organized criticality: emergent complex behavior in physical and
  biological systems}
\bvolume{10}.
\bpublisher{Cambridge university press}.
\end{bbook}
\endbibitem

\bibitem{JLS12}
\begin{barticle}[author]
\bauthor{\bsnm{Jerison},~\bfnm{David}\binits{D.}},
  \bauthor{\bsnm{Levine},~\bfnm{Lionel}\binits{L.}} \AND
  \bauthor{\bsnm{Sheffield},~\bfnm{Scott}\binits{S.}}
(\byear{2012}).
\btitle{Logarithmic fluctuations for internal {DLA}}.
\bjournal{J. Amer. Math. Soc.}
\bvolume{25}
\bpages{271--301}.
\bdoi{10.1090/S0894-0347-2011-00716-9}
\bmrnumber{2833484}
\end{barticle}
\endbibitem

\bibitem{JJ18}
\begin{barticle}[author]
\bauthor{\bsnm{Johnson},~\bfnm{Tobias}\binits{T.}} \AND
  \bauthor{\bsnm{Junge},~\bfnm{Matthew}\binits{M.}}
(\byear{2018}).
\btitle{Stochastic orders and the frog model}.
\bjournal{Annales de l'Institut Henri Poincar{\'e}, Probabilit{\'e}s et
  Statistiques}
\bvolume{54}
\bpages{1013--1030}.
\end{barticle}
\endbibitem

\bibitem{KS05}
\begin{barticle}[author]
\bauthor{\bsnm{Kesten},~\bfnm{Harry}\binits{H.}} \AND
  \bauthor{\bsnm{Sidoravicius},~\bfnm{Vladas}\binits{V.}}
(\byear{2005}).
\btitle{The spread of a rumor or infection in a moving population}.
\bjournal{The annals of probability}
\bvolume{33}
\bpages{2402--2462}.
\end{barticle}
\endbibitem

\bibitem{LS21}
\begin{barticle}[author]
\bauthor{\bsnm{Levine},~\bfnm{Lionel}\binits{L.}} \AND
  \bauthor{\bsnm{Silvestri},~\bfnm{Vittoria}\binits{V.}}
(\byear{2021}).
\btitle{How far do activated random walkers spread from a single source?}
\bjournal{J. Stat. Phys.}
\bvolume{185}
\bpages{Paper No. 18, 27}.
\bdoi{10.1007/s10955-021-02836-9}
\bmrnumber{4334780}
\end{barticle}
\endbibitem

\bibitem{LP16}
\begin{bbook}[author]
\bauthor{\bsnm{Lyons},~\bfnm{Russell}\binits{R.}} \AND
  \bauthor{\bsnm{Peres},~\bfnm{Yuval}\binits{Y.}}
(\byear{2016}).
\btitle{Probability on trees and networks}.
\bseries{Cambridge Series in Statistical and Probabilistic Mathematics}
\bvolume{42}.
\bpublisher{Cambridge University Press, New York}.
\bdoi{10.1017/9781316672815}
\bmrnumber{3616205}
\end{bbook}
\endbibitem

\bibitem{Manna91}
\begin{barticle}[author]
\bauthor{\bsnm{Manna},~\bfnm{Subhrangshu~Sekhar}\binits{S.~S.}}
(\byear{1991}).
\btitle{Two-state model of self-organized criticality}.
\bjournal{Journal of Physics A: Mathematical and General}
\bvolume{24}
\bpages{L363}.
\end{barticle}
\endbibitem

\bibitem{Rolla08}
\begin{barticle}[author]
\bauthor{\bsnm{Rolla},~\bfnm{Leonardo~T}\binits{L.~T.}}
(\byear{2008}).
\btitle{Generalized hammersley process and phase transition for activated
  random walk models}.
\bjournal{PhD thesis, arXiv:0812.2473}.
\end{barticle}
\endbibitem

\bibitem{Rolla20}
\begin{barticle}[author]
\bauthor{\bsnm{Rolla},~\bfnm{Leonardo~T.}\binits{L.~T.}}
(\byear{2020}).
\btitle{Activated random walks on {$\mathbb{Z}^d$}}.
\bjournal{Probab. Surv.}
\bvolume{17}
\bpages{478--544}.
\bdoi{10.1214/19-PS339}
\bmrnumber{4152668}
\end{barticle}
\endbibitem

\bibitem{RS12}
\begin{barticle}[author]
\bauthor{\bsnm{Rolla},~\bfnm{Leonardo~T.}\binits{L.~T.}} \AND
  \bauthor{\bsnm{Sidoravicius},~\bfnm{Vladas}\binits{V.}}
(\byear{2012}).
\btitle{Absorbing-state phase transition for driven-dissipative stochastic
  dynamics on {${\mathbb Z}$}}.
\bjournal{Invent. Math.}
\bvolume{188}
\bpages{127--150}.
\bdoi{10.1007/s00222-011-0344-5}
\bmrnumber{2897694}
\end{barticle}
\endbibitem

\bibitem{RSZ19}
\begin{barticle}[author]
\bauthor{\bsnm{Rolla},~\bfnm{Leonardo~T.}\binits{L.~T.}},
  \bauthor{\bsnm{Sidoravicius},~\bfnm{Vladas}\binits{V.}} \AND
  \bauthor{\bsnm{Zindy},~\bfnm{Olivier}\binits{O.}}
(\byear{2019}).
\btitle{Universality and sharpness in activated random walks}.
\bjournal{Ann. Henri Poincar\'{e}}
\bvolume{20}
\bpages{1823--1835}.
\bdoi{10.1007/s00023-019-00797-0}
\bmrnumber{3956161}
\end{barticle}
\endbibitem

\bibitem{RT18}
\begin{barticle}[author]
\bauthor{\bsnm{Rolla},~\bfnm{L.~T.}\binits{L.~T.}} \AND
  \bauthor{\bsnm{Tournier},~\bfnm{L.}\binits{L.}}
(\byear{2018}).
\btitle{Non-fixation for biased activated random walks}.
\bjournal{Ann. Inst. Henri Poincar\'{e} Probab. Stat.}
\bvolume{54}
\bpages{938--951}.
\bdoi{10.1214/17-AIHP827}
\bmrnumber{3795072}
\end{barticle}
\endbibitem

\bibitem{Shellef10}
\begin{barticle}[author]
\bauthor{\bsnm{Shellef},~\bfnm{Eric}\binits{E.}}
(\byear{2010}).
\btitle{Nonfixation for activated random walks}.
\bjournal{ALEA Lat. Am. J. Probab. Math. Stat.}
\bvolume{7}
\bpages{137--149}.
\bmrnumber{2651824}
\end{barticle}
\endbibitem

\bibitem{ST17}
\begin{barticle}[author]
\bauthor{\bsnm{Sidoravicius},~\bfnm{Vladas}\binits{V.}} \AND
  \bauthor{\bsnm{Teixeira},~\bfnm{Augusto}\binits{A.}}
(\byear{2017}).
\btitle{Absorbing-state transition for stochastic sandpiles and activated
  random walks}.
\bjournal{Electron. J. Probab.}
\bvolume{22}
\bpages{Paper No. 33, 35}.
\bdoi{10.1214/17-EJP50}
\bmrnumber{3646059}
\end{barticle}
\endbibitem

\bibitem{ST18}
\begin{barticle}[author]
\bauthor{\bsnm{Stauffer},~\bfnm{Alexandre}\binits{A.}} \AND
  \bauthor{\bsnm{Taggi},~\bfnm{Lorenzo}\binits{L.}}
(\byear{2018}).
\btitle{Critical density of activated random walks on transitive graphs}.
\bjournal{Ann. Probab.}
\bvolume{46}
\bpages{2190--2220}.
\bdoi{10.1214/17-AOP1224}
\bmrnumber{3813989}
\end{barticle}
\endbibitem

\bibitem{Taggi16}
\begin{barticle}[author]
\bauthor{\bsnm{Taggi},~\bfnm{Lorenzo}\binits{L.}}
(\byear{2016}).
\btitle{Absorbing-state phase transition in biased activated random walk}.
\bjournal{Electron. J. Probab.}
\bvolume{21}
\bpages{Paper No. 13, 15}.
\bdoi{10.1214/16-EJP4275}
\bmrnumber{3485355}
\end{barticle}
\endbibitem

\bibitem{Taggi19}
\begin{barticle}[author]
\bauthor{\bsnm{Taggi},~\bfnm{Lorenzo}\binits{L.}}
(\byear{2019}).
\btitle{Active phase for activated random walks on {$\mathbb{Z}^d$},
  {$d\geq3$}, with density less than one and arbitrary sleeping rate}.
\bjournal{Ann. Inst. Henri Poincar\'{e} Probab. Stat.}
\bvolume{55}
\bpages{1751--1764}.
\bdoi{10.1214/18-aihp933}
\bmrnumber{4010950}
\end{barticle}
\endbibitem

\end{thebibliography}


\end{document}